\documentclass[11pt]{article}

\usepackage{ifpdf}
\ifpdf 
    \usepackage[pdftex]{graphicx}   
    \usepackage[pdftex,     
            plainpages=false,   
            breaklinks=true,    
            colorlinks=true,
            pdftitle=My Document
            pdfauthor=My Good Self
           ]{hyperref} 
    \usepackage{thumbpdf}
    \usepackage{mathtools,xparse}
    \usepackage{amsmath}
	\usepackage{amsthm}    
    \usepackage{subfigure}
      \usepackage{float}
    \newtheorem{theorem}{Theorem}[section]
\newtheorem{lemma}{Lemma}[section]
    \newtheorem{remark}{Remark}[section]
    \usepackage{mathrsfs}
     \usepackage{thumbpdf}
    \usepackage{indentfirst}
\usepackage{amsmath}
\usepackage{titlesec}
\usepackage{setspace}
\usepackage{geometry}
\usepackage{scalerel,stackengine}
\DeclareMathOperator{\sech}{sech}
\DeclareMathOperator{\sign}{sgn}
\stackMath
\newcommand\widecheck[1]{%
\savestack{\tmpbox}{\stretchto{%
  \scaleto{%
    \scalerel*[\widthof{\ensuremath{#1}}]{\kern-.6pt\bigwedge\kern-.6pt}%
    {\rule[-\textheight/2]{1ex}{\textheight}}
  }{\textheight}%
}{0.5ex}}%
\stackon[1pt]{#1}{\scalebox{-1}{\tmpbox}}%
}
\newcommand{\RNum}[1]{\uppercase\expandafter{\romannumeral #1\relax}}
\DeclarePairedDelimiter{\norm}{\lVert}{\rVert}
\NewDocumentCommand{\normL}{ s O{} m }{%
  \IfBooleanTF{#1}{\norm*{#3}}{\norm[#2]{#3}}_{L_2(\Omega)}%
}
\else 
    \usepackage{graphicx}       
    \usepackage{hyperref}       
\fi 
\usepackage{enumitem}

\usepackage{color}

\title{Local discontinuous Galerkin methods for the abcd nonlinear Boussinesq system}
\author{	Jiawei Sun\footnote{Department of Mathematics, The Ohio State University,
		Columbus, OH 43210, USA. E-mail: sun.2261@buckeyemail.osu.edu.} \and
		Shusen Xie\footnote{School of Mathematical Sciences, Ocean University of China, Qingdao, China. E-Mail: shusenxie@ouc.edu.cn.} 
		\and Yulong Xing\footnote{Department of Mathematics, The Ohio State University,
		Columbus, OH 43210, USA. E-mail: xing.205@osu.edu. }} 

\date{}

\begin{document}
\maketitle

\begin{abstract}
  Boussinesq type equations have been widely studied to model the surface 
  water wave.  In this paper, we consider the abcd Boussinesq system which is a 
  family of Boussinesq type equations including many well-known models such as the classical Boussinesq system, 
  BBM-BBM system,
  Bona-Smith system etc. We propose local discontinuous Galerkin (LDG) 
  methods, with carefully chosen numerical fluxes, to numerically solve this abcd Boussinesq system. The main focus of this paper is 
  to rigorously establish a priori error estimate of the proposed LDG methods for a wide range of the parameters $a,b,c,d$. 
  Numerical experiments are shown to test the convergence rates, and to demonstrate that the proposed
  methods can simulate the head-on collision of traveling wave and finite time blow-up behavior well. 
\end{abstract}

\section{Introduction} \setcounter{equation}{0} \setcounter{table}{0} \setcounter{figure}{0}

In this paper, we present and analyze the discontinuous Galerkin (DG) method for a family of nonlinear dispersive water wave models: the abcd Boussinesq system of the form 
\begin{equation}\label{1}
  \begin{cases}
\eta_t+u_x+(\eta u)_x+au_{xxx}-b\eta_{xxt}=0, &\\
u_t+\eta_x+\big(\frac{u^2}{2}\big)_x+c\eta_{xxx}-du_{xxt}=0,&
\end{cases}
\end{equation}
where $\eta$ denotes the proportional deviation of the free surface from its 
rest position, and $u$ represents the proportional horizontal velocity. 
Here $a,~b,~c,~d$ are constant satisfying $a+b+c+d={1}/{3}$. When $a=b=c=d=0$, this model reduces to
the well-known nonlinear shallow water equations (by replacing $1+\eta$ with water height $h$). 
The goal of the paper is to provide the numerical approximations of \eqref{1} via DG methods, and establish 
rigorous error estimate on their convergence rate for a wide range of the parameters $a,~b,~c,~d$. 
Periodic boundary conditions are considered for simplicity, and similar analysis can be extended to other types of boundary conditions.

Nonlinear dispersive water wave models have wide applications in engineering applications. 
One typical example frequently encountered in practical is the propagation of water wave 
with small amplitude and long wavelength.
For such applications, the Boussinesq approximation for water waves is developed by Boussinesq \cite{Boussinesq1871},
which leads to the Boussinesq type equations. These models are derived from the full Euler equations, and have been frequently used 
in modeling water waves in shallow seas and harbors. 
Many different Boussinesq type equations have been proposed and studied, including the ``good'' Boussinesq equation and the improved
Boussinesq equation. In \cite{Bona1}, Bona et al. introduced a 
generalized case of the Boussinesq equations, which is called the abcd
Boussinesq system \eqref{1}. 
The constant parameters $a,b,c,d$ in \eqref{1} satisfy the following relation:
\begin{align}
	a + b = \frac{1}{2}\left(\theta^2 - \frac{1}{3} \right), \quad c + d &= \frac{1}{2}(1 - \theta^2) \geq 0, 
	  \label{abcdparam}
\end{align}
with the scaled height $\theta\in [0,1]$. $\theta = 0$ represents the bottom of the channel and $\theta = 1$ is the free surface.
Based on different choices of the parameters $a,b,c,d$, this family of equations reduces to some well-known water wave models,
including the classical Boussinesq system ($a=b=c=0$, $d=1/3$), the coupled Benjamin-Bona-Mahony (BBM) system ($a=c=0$, $b=d=1/6$),
the coupled Korteweg-de Vries (KdV) system ($b=d=0$, $a=c=1/6$), and the Bona-Smith system ($a=0$, $b=d$) etc.

Since the introduction of the abcd Boussinesq system, there have been many theoretical and numerical studies of this model. In \cite{Bona2}, Bona, Chen and Saut showed that this system is (locally-in-time) well posed in the following cases:    
$$
a,~c\le 0 ~\text{and}~b,~d\ge0 \hskip20mm \text{or} \hskip20mm a=c\ge0 \text{ and }b,~d\ge0.
$$
The global existence of the solutions can only be shown for some special cases of the parameters $a,b,c,d$. 
For instance in \cite{Bona2}, Bona, Chen and Saut confirmed the global existence for the Bona-Smith system 
with the assumption that the initial value satisfies some smallness condition.
In \cite{Amick}, Amick proved global existence for the classical Boussinesq system. 
In contrast, previous works show that solutions to some models, for example the coupled BBM system, 
could blow up in finite time when $1+\eta<0$. 

There have been many studies on various numerical methods for the special cases of the abcd Boussinesq system \eqref{1}. 
In 1965, Peregrine in \cite{Peregrine} developed a straight-forward finite 
difference approximation to the classical Boussinesq system in 
order to study the propagation of undular bore. In \cite{BonaChen1997}, Bona and 
Chen rewrote the coupled BBM system into a system of integral 
equations, based on which they proposed a numerical scheme with fourth order 
accuracy in both time and space. They also utilized this scheme to simulate the 
collision of solitary waves. For the coupled KdV-KdV equations, Bona, Dougalis and Mitsotakis \cite{BDM} 
implemented the standard Galerkin semidiscretization with periodic smooth spline basis of order $k$, and numerically obtained 
$k+1$-th order of accuracy for spatial convergence rate. 
 Besides these numerical works for the special cases of the system \eqref{1}, there have been several works to tackle 
the general family of abcd Boussinesq system with a large class of parameters $a,b,c,d$.
 In \cite{Antonopoulos}, Antonopoulos, Dougalis and Mitsotakis extended the scheme in \cite{BDM} onto general $a,b,c,d$ 
 parameters and achieved optimal order error estimate for the following 
 categories: 
\begin{enumerate}[label=\textsc{(C\arabic*)},wide]
   \item $a,~c<0,~b,~d>0$;
      \item $b,~d>0$, $a=c>0$;
  \item either $b>0$, $d=0$ or $b=0$, $d>0$, and either $a,~c<0$ or $a=c\ge0$;
   \item $b,~d>0$, $ac=0$ and $a,~c\leq 0$. 

 \end{enumerate}
%
Recently, C. Burtea and C. Court\`{e}s \cite{CB} presented fully discrete finite 
volume scheme for the Boussinesq system (\ref{1}), and 
proved the convergence rate of first order in time and second order in space (when $bd>0$, first order in space when $bd=0$).
The analysis is provided for a broad range of the parameters $a,b,c,d$, which satisfies 
$$
a,~c\leq0~ \text{and} ~b,~d\geq 0,
$$
excluding the following cases: $a=b=0,d\geq 0,c<0$; $a=d=0, b>0,c<0$; $a=b=d=0, c<0$ and $b=d=0, a<0, c<0$.

High order DG methods are considered in this paper to provide efficient numerical approximation of the abcd Boussinesq system (\ref{1}). 
DG method is a class of finite element methods, which uses completely discontinuous basis functions, i.e. piecewise polynomials. 
It inherits the advantages of both finite element and finite volume methods, and has been widely applied in practical application in the last decade. 
We refer to \cite{CHS1990, CLS1989, CS1989}, and \cite{CKS2000} for a historic review. The DG method has various advantages, such as hp-adaptivity flexibility, the local conservativity, efficient parallel implementation, easy coordination with finite volume techniques, the ability of easy handling of complicated geometries and boundary conditions and so on.
For equations with high order spatial derivatives, several types of DG methods have been developed, including the interior penalty DG methods, local discontinuous Galerkin (LDG) methods, ultra-weak DG methods, hybridizable DG methods and many others. 
Among them, the LDG method was developed by Cockburn and Shu in \cite{CC}, and the main idea was to rewrite the original equation into 
a system of first order equations by introducing auxiliary variables, and then discretize it with the standard DG methods. With carefully chosen
numerical fluxes, the stability and error estimate of the LDG methods have been studied for many models with high order spatial derivatives, 
and we refer to the review paper \cite{XS2010} for the development and applications of LDG methods.  


The LDG method has been applied to the scalar generalized KdV equation 
\[u_t+f(u)_x+\epsilon u_{xxx}=0,\]
by Xu and Shu in \cite {YS}, and the sub-optimal error estimate of the $k+1/2$-th order was obtained. 
Energy conserving ultra-weak DG and LDG methods for the KdV equation were studied in \cite{BCKX2013,KX2016},
where energy conserving methods that exactly preserve the discrete energy were shown to provide more accurate numerical approximation
in the long time simulation. 
In \cite{LSXC2020}, energy conserving LDG methods for the improved Boussinesq equation with the optimal error estimate were studied. 
In \cite{JX}, both energy conserving and energy dissipative LDG methods are designed for the coupled BBM system. 
The optimal error estimate of the LDG methods is carried out only for the linearized BBM system. 
It is usually challenging to generalize the error estimate of the scalar nonlinear equation to a nonlinear system. 
In \cite{QZ}, the sub-optimal error estimate is acquired for a family of symmetrizable first order system. 
In particular, the nonlinear shallow water equation, i.e., system (\ref{1}) with $a=b=c=d=0$, is a symmetrizable system. 
The error estimate of the nonlinear coupled BBM system was not available in \cite{JX}, due to the presence of both nonlinear first order
derivative and third order mixed derivative terms. Recently, an optimal error estimate of the energy conserving and energy dissipative
LDG method for the scalar BBM equation was presented in \cite{XL}, by a novel observation to discover the connection between the error of 
the auxiliary and primary variables. 
In this paper, we consider the abcd Boussinesq system (\ref{1}) and first rewrite the model into several first order differential equations. 
The main contribution of this paper is to design high order accurate LDG methods, with carefully chosen numerical fluxes, 
for the abcd Boussinesq system (\ref{1}), 
and to establish rigorous error estimate of the proposed LDG methods for a wide range of 
the parameters $a,b,c,d$, including the coupled nonlinear BBM system that was not analyzed in \cite{JX}.
More specifically, we will provide the error estimate for the cases (C1)-(C4) of the parameters $a,b,c,d$,
similar to those studied in \cite{Antonopoulos}, but different analyses are needed to handle the numerical fluxes and boundaries terms
arising from the DG approximation. 
One important tool in our analysis is the connection between the error of the auxiliary and primary variables, discovered in \cite{XL}. 
Numerical examples are also provided to demonstrate the convergence rates and to show that the proposed method has the capability to simulate the head-on collision and finite time blow-up behavior well.


The structure of this paper is as follows. Section \ref{nota} provides an 
introduction of the notations used throughout the paper. In Section \ref{LDG}, 
the proposed LDG scheme and the corresponding numerical fluxes are presented. Section \ref{thms} 
devotes to the main theoretical results on the error estimate of the proposed methods. Depending on different cases of the
parameters $a,~b,~c,~d$, several theorems are provided to discuss the 
convergence rate of the numerical methods. Numerical results are provided in Section \ref{num},
which consists of two tests: accuracy test and wave collisions. Conclusion remarks are given in Section \ref{conclusion}.

\section{Notations and projections}\label{nota} \setcounter{equation}{0} \setcounter{table}{0} \setcounter{figure}{0}

The computational domain $I$ can be divided into subintervals, denoted by $I_j=[x_{j-\frac{1}{2}},x_{j+\frac{1}{2}}]$ for $1,2,\cdots,N$. The 
center of each cell is $x_j=\frac{1}{2}(x_{j-\frac{1}{2}}+x_{j+\frac{1}{2}})$ 
and the mesh size of each cell is $h_j=x_{j+\frac{1}{2}}-x_{j-\frac{1}{2}}$, with the maximum mesh size denoted by $h=\max h_j$. 
We further assume that the mesh to be quasi-uniform, i.e., the ratio ${h}/{h_j}$ always stay bounded for all $j$ during mesh refinement. Denote $H^k(I_j)$ as the standard Hilbert space, and $P^k(I_j)$ as the space of polynomials of degree up to $k$ on the cell $I_j$. The piecewise polynomial space $V_h^k$ is defined as
\[V^k_h=\{v: v|_{I_j}\in P^k(I_j),\, j=1,2,\cdots,N\}.\]

The numerical solutions that approximate the unknown functions $u$, $\eta$ are denoted by $u_h$ and $\eta_h$ in the solution space $V_h^k$. For ease of presentation, the following notations are introduced: for any function $g(x)$, and $x_\lambda$ inside the computational domain where $\lambda$
 comes from some index set, define
\[g_{\lambda}=g(x_\lambda),~~g^\pm _{\lambda}=\lim_{x\to x_{\lambda}^\pm}g(x), ~~~[g]_{\lambda}=g_{\lambda}^+-g_{\lambda}^-, ~~~\{g\}_{\lambda}=\frac{1}{2}(g^+_{\lambda}+g^-_{\lambda}).\] 
We introduce the inner product notation $(v, w)_{I_j} = \int_{I_j} v w ~dx$, and $(v, w)_I = \sum_{j} (v, w)_{I_j}$. 
In our analysis, we use the following norms:
\begin{itemize}
  \item $\norm{v}=\norm{v}_{L^2}$, and $\norm{v}_\infty$ denote the $L^2$ norm and $L^\infty$ norm of $v$;
  \item $\norm{v}_{\Gamma_h}=\sqrt{\sum\limits^N_{j=1}(v^{\pm}_{j+ 
  \frac{1}{2}})^2},$ where $\Gamma_h$ denotes the set of boundary points of all the elements $I_j$;
  \item $|[v]|=\sqrt{\sum\limits^N_{j=1}[v]^2_{j+\frac{1}{2}}}$.
\end{itemize}
Throughout this paper, $C$ denotes a generic positive constant independent of the spatial and temporal step sizes $h$ and $\Delta t$, which may have different values at different occasions.

Next, we discuss some projection operators that will be used in the error estimate analysis, and present their approximation property.
We define the standard $L^2$ projection $\mathcal{P}$ in the following way: for any element $I_j$,
\[ \int_{I_j}(\mathcal{P}g-g(x))v(x)dx=0,~~\forall v(x)\in  P^k(I_j).\] 
The Radau projection $\mathcal{P}^\pm$ is defined such that for any element $I_j$, we have
\[\int_{I_j}(\mathcal{P}^{+}g-g(x))v(x)dx=0,~~\forall v(x)\in P^{k-1}(I_j),   \quad \text{and} \quad (\mathcal{P}^+g)^+_{j-\frac{1}{2}}=g_{j-\frac{1}{2}}, \]
\[\int_{I_j}(\mathcal{P}^{-}g-g(x))v(x)dx=0,~~\forall v(x)\in P^{k-1}(I_j),   \quad \text{and} \quad (\mathcal{P}^-g)^-_{j+\frac{1}{2}}=g_{j+\frac{1}{2}}. \]
The following approximation property of these projections has been studied in \cite{PC} and is listed below:
\begin{equation}\label{interpolation}
  \norm{\mathcal{P}^*g-g}+h\norm{\mathcal{P}^*g-g}_{\infty}+{h}^{\frac12}\norm{\mathcal{P}^*g-g}_{\Gamma_h}\le Ch^{k+1}.
\end{equation}
where $\mathcal{P}^*=\mathcal{P}, \mathcal{P}^\pm$. 

Another useful tool to be used throughout this paper is the \textit{inverse property}: for any $g\in V_h^k$, we have
\begin{equation}\label{inverse}
  \norm{g_x}\le Ch^{-1}\norm{g}, \qquad 
  \norm{g}_{\Gamma_h}\le Ch^{-\frac{1}{2}}\norm{g}.
\end{equation}

\section{The local discontinuous Galerkin method}\label{LDG} \setcounter{equation}{0} \setcounter{table}{0} \setcounter{figure}{0}

Following the idea of LDG method, we first rewrite the system of equation (\ref{1}) as a system of first order equations
\begin{subequations} \label{ldg_method}
\begin{alignat}{3}
&\eta_t+u_x+(\eta u)_x+aq_x-b\theta_t=0,	 \quad  &w=\eta_x,		\quad  &\theta=\zeta=w_x,\\
&u_t+\eta_x+f(u)_x+c\zeta_x-dp_t=0, 		\quad  &v=u_x,		\quad  &p=q=v_x,
\end{alignat}
\end{subequations}
where $f(u)={u^2}/{2}$.
Denote the nonlinear flux term as $F(\eta, u)=(u+\eta u, \eta+u^2/2)^T$, and introduce its approximation via the standard Lax-Friedrichs numerical flux: 
\begin{equation}
\widehat{F}=\left(\begin{matrix} 
\displaystyle \{u_h\}+\{\eta_h u_h\}-\frac{\alpha}{2}[\eta_h] \\
\displaystyle \{\eta_h\}+\frac{1}{2}\{u^2_h\}-\frac{\alpha}{2}[u_h]\end{matrix}\right)
=:\left(
\begin{matrix} \displaystyle\widehat{F}_1\\ \displaystyle\widehat{F}_2\end{matrix}\right),
\end{equation}
where $\alpha=\max(|u|+\sqrt{|1+\eta|})$ with the maximum taking over the whole region.
The LDG method for approximating \eqref{ldg_method} is given as follows: find $\eta_h, u_h, v_{h}, w_{h},  p_h, \theta_h,q_h,\zeta_h \in V^k_h$, 
such that the following equations 
\begin{eqnarray}
  &&\int_{I_j}(\eta_h-b\theta_h)_t\rho dx-\int_{I_j}(u_h+u_h\eta_h+aq_h)\rho_x 
  dx \notag 
  +(\widehat{F}_1+a\widehat{q}_{h})\rho^-_{j+\frac{1}{2}}-
  (\widehat{F}_1+a\widehat{q}_{h})\rho^+_{j-\frac{1}{2}}=0, \notag \\ 
  &&\int_{I_j}(u_h-dp_h)_t\widetilde{\rho}dx-
  \int_{I_j}(\eta_h+f(u_h))+c\zeta_h)\widetilde{\rho}_xdx 
  +(\widehat{F}_2+c \widehat{\zeta}_{h}){\widetilde{\rho}}^-_{j+\frac{1}{2}}
  -(\widehat{F}_2+c\widehat{\zeta}_{h})){\widetilde{\rho}}^+_{j-\frac{1}{2}}=0,  \notag \\
  &&\int_{I_j}v_h\phi dx+\int_{I_j}u_h \phi_{x} dx-\widehat{u}_h\phi^-_{j+\frac{1}{2}}
  +\widehat{u}_h\phi^+_{j-\frac{1}{2}}=0,~~  \notag\\
  &&\int_{I_j}p_h\psi dx+\int_{I_j}v_{h} \psi_x dx-\widehat{v}_{h}\psi^-_{j+\frac{1}{2}}
  +\widehat{v}_{h}\psi^+_{j-\frac{1}{2}}=0, \notag \\
  &&\int_{I_j}q_h\varphi dx+\int_{I_j}v_{h} \varphi_{x} dx-\widecheck{v}_{h}\varphi^-_{j+\frac{1}{2}}
  +\widecheck{v}_{h}\varphi^+_{j-\frac{1}{2}}=0,\notag \\
  &&\int_{I_j}w_h\widetilde{\phi} dx+\int_{I_j}\eta_h\widetilde{\phi}_x dx
  -\widehat{\eta}_h\widetilde{\phi}^-_{j+\frac{1}{2}}
  +\widehat{\eta}_h\widetilde{\phi}^+_{j-\frac{1}{2}}=0,   \label{scheme}\\
  &&\int_{I_j}\theta_h\widetilde{\psi} dx+\int_{I_j}w_h \widetilde{\psi}_x dx
  -\widehat{w}_h\widetilde{\psi}^-_{j+\frac{1}{2}}
  +\widehat{w}_h\widetilde{\psi}^+_{j-\frac{1}{2}}=0, \notag\\
  &&\int_{I_j}\zeta_h\widetilde{\varphi} dx+\int_{I_j}w_{h} \widetilde{\varphi}_{x} dx
  -\widecheck{w}_{h}\widetilde{\varphi}^-_{j+\frac{1}{2}}
  +\widecheck{w}_{h}\widetilde{\varphi}^+_{j-\frac{1}{2}}=0, \notag
\end{eqnarray}
hold for all the test functions $\rho, \phi,\psi,\widetilde{\rho}, \widetilde{\phi},$ $ \widetilde{\psi}, \varphi,\widetilde{\varphi}\in V^k_h$.
 We choose the following numerical fluxes
\begin{eqnarray}\label{ss1}
  \begin{split}
    \widehat{q}_{h}=q_h^-,~~\widehat{\eta}_h=\eta^+_h,~~\widehat{w}_h=w^-_h,~~\widecheck{w}_{h}=w_{h}^-+\sign(a)\frac{\lambda[v_h]}{2},~~ \\
    \widehat{\zeta}_{h}=\zeta_h^-,~~\widehat{u}_h=u^+_h,~~\widehat{v}_h=v^-_h,~~\widecheck{v}_{h}=v^+_h+\sign(c)\frac{\lambda[w_h]}{2},~~
  \end{split}
\end{eqnarray}
  at the cell interfaces, where $\sign(a)$ represents the sign of parameter $a$ and $\lambda$ is a positive constant. Without loss of generality, $\lambda$ is taken as $1$ in the remainder of this paper. 

Notice that this choice of numerical fluxes is not unique. To ensure stability, the essential guideline is to 
take $\widehat{q}_{h}$ and $\widehat{u}_h$ from the opposite direction, $~\widehat{\zeta}_{h}$ and $\widehat{\eta}_h$ 
from the opposite direction. Another remark is that $w_x$ and $v_x$ are both approximated by two variables ($\theta$, $\zeta$ and $p$, $q$),
which leads to two different fluxes for $v_h$ and $w_h$. This will be useful in our error estimate.
For short hand notation, we introduce the DG discretization operator $\mathcal{A}_j$, defined in the following way:
\begin{eqnarray*}
  &\mathcal{A}_j( f,g; \hat{f})=\int_{I_j}fg_xdx-
  \big( \hat{f}g^-\big)_{j+\frac{1}{2}}+\big(\hat{f}g^+\big)_{j-\frac{1}{2}},
\end{eqnarray*}
for any $g \in V_h^k$ and $f\in H^{k+1}(I)$, with $\hat{f}$ being the numerical fluxes. 

The initial data $u_h(x,0),~\eta_h(x,0),~v_h(x,0),~w_h(x,0),~p_h(x,0)$ are chosen in the following way to ensure the good approximation property
of these initial data at $t=0$. Given initial condition $u(x,0)=u_0(x)$ and $\eta(0,x)=\eta_0(x)$, define 
  \[u_h(x,0)=\mathcal{P}^+u(x,0),~~\eta_h(x,0)=\mathcal{P}^+\eta(x,0),\]
  and let $v_h(x,0),~w_h(x,0),~p_h(x,0)$ be the solutions to the following 
  equations: 
  \begin{eqnarray}
  &&\int_{I_j}v_h\phi dx+\int_{I_j}u_h \phi_{x} dx-\widehat{u}_h\phi^-_{j+\frac{1}{2}}
  +\widehat{u}_h\phi^+_{j-\frac{1}{2}}=0,~~  \label{v}\\
  &&\int_{I_j}w_h\widetilde{\phi} dx+\int_{I_j}\eta_h\widetilde{\phi}_x dx
  -\widehat{\eta}_h\widetilde{\phi}^-_{j+\frac{1}{2}}
  +\widehat{\eta}_h\widetilde{\phi}^+_{j-\frac{1}{2}}=0, \label{w}\\
  &&\int_{I_j}p_h\psi dx+\int_{I_j}v_{h} \psi_x dx-\widehat{v}_{h}\psi^-_{j+\frac{1}{2}}
  +\widehat{v}_{h}\psi^+_{j-\frac{1}{2}}=0. \label{p}
\end{eqnarray}
\begin{lemma}\label{initial}
The numerical initial conditions satisfy
  \begin{align*}
  &\norm{u(x,0)-u_h(x,0)}^2+\norm{v(x,0)-v_h(x,0)}^2\le 
  Ch^{2k+2},\\
  &\norm{\eta(x,0)-\eta_h(x,0)}^2+\norm{w(x,0)-w_h(x,0)}^2\le Ch^{2k+2},\\
  &\norm{p(x,0)-p_h(x,0)}^2\le Ch^{2k}.
  \end{align*}
\end{lemma}
\begin{proof} 
By the projection property (\ref{interpolation}), we have
\begin{align*}
&\norm{u(x,0)-u_h(x,0)}=\norm{\mathcal{P}^+u(x,0)-u(x,0)}\le Ch^{k+1},\\
&\norm{\eta(x,0)-\eta_h(x,0)}=\norm{\mathcal{P}^+\eta(x,0)-\eta(x,0)}\le Ch^{k+1}.
\end{align*}
The equation (\ref{v}) leads to the error equation
\begin{align*}
  \int_{I_j}(v-v_h)\phi dx+\int_{I_j}(u-\mathcal{P}^+u)\phi_x dx
  -\big((u-\mathcal{P}^+u)^+\phi^-\big)_{j+\frac{1}{2}}
  +\big((u-\mathcal{P}^+u)^+\phi^+\big)_{j-\frac{1}{2}}=0,
\end{align*}
which reduces to
\[\int_{I_j}(v-v_h)\phi dx=0,\]
as a result of the projection $\mathcal{P}^+$.
Choose $\phi=\mathcal{P}v-v_h$, and decompose $v-v_h=(\mathcal{P}v-v_h)-(\mathcal{P}v-v)$. 
We have
\[\int_{I_j}(\mathcal{P}v-v_h)^2dx=\int_{I_j}(\mathcal{P}v-v)(\mathcal{P}v-v_h)dx=0.\]
Therefore we conclude $v_h(x,0)=\mathcal{P}v(x,0)$, and
$\norm{v(x,0)-v_h(x,0)}\le Ch^{k+1}$.
Similarly, we can show that
$\norm{w(x,0)-w_h(x,0)}\le Ch^{k+1}$.
The equation (\ref{p}) leads to
 \[\int_{I_j}(p-p_h)\psi dx+\int_{I_j}(v-\mathcal{P}v)\psi_xdx
 -\big((v-\mathcal{P}v)^-\psi^-\big)_{j+\frac{1}{2}}+
 \big((v-\mathcal{P}v)^-\psi^+\big)_{j-\frac{1}{2}}=0.\]
By choosing $\psi=\mathcal{P}p-p_h$ and decomposing $p-p_h=(\mathcal{P}p-p_h)-(\mathcal{P}p-p)$, we have 
 \[\int_{I_j}(\mathcal{P}p-p_h)^2dx=\Big((\mathcal{P}v-v)^-[\mathcal{P}p-p_h]\Big)_{j+\frac{1}{2}}.\]
Summing over all cells $I_j$, utilizing the inverse property (\ref{inverse}) and 
 the projection property (\ref{interpolation}), we obtain
 \begin{align*}
 &  \int_I(\mathcal{P}p-p_h)^2dx=\sum^N_{j=1}\Big((\mathcal{P}v-v)^-[\mathcal{P}p-p_h]\Big)_{j+\frac{1}{2}} \\
 & \qquad
 =\sum^N_{j=1}\Big(h^{-\frac{1}{2}}(\mathcal{P}v-v)^-h^{\frac{1}{2}}[\mathcal{P}p-p_h]\Big)_{j+\frac{1}{2}}
 \le \frac{1}{2}\int_I (\mathcal{P}p-p_h)^2dx+Ch^{2k},
 \end{align*}
which leads to $\norm{p(x,0)-p_h(x,0)}^2\le Ch^{2k}$ and finishes the proof.
\end{proof}
Note that the (suboptimal) initial error of $p_h$ is only used in the proof of Theorem \ref{thm4}, and the error estimate in that theorem is also suboptimal. One may construct a different choice of the numerical initial values to ensure the optimal error of $p_h$ as well, see for instance \cite{HX2014}.

\section{Error estimate}\label{thms}\setcounter{equation}{0} \setcounter{table}{0} \setcounter{figure}{0}

In this section, we derive the error estimate of the proposed semi-discrete LDG method \eqref{scheme}. The analysis will be separated into five subsections, each corresponds to one of five different cases of the parameters $a$, $b$, $c$ and $d$ to be studied.

We start by defining the error terms and presenting the error equations. For any function $g$ and its numerical approximation $g_h$,
the error can be separated into two components
\begin{eqnarray}\label{ss}
	e^g := g-g_h = \xi^g-\epsilon^g,
\end{eqnarray}
where 
$$ \xi^g=\mathcal{P}^*g-g_h,\qquad \epsilon^{g}=\mathcal{P}^*g-g,$$
denote the error between a particular projection $\mathcal{P}^*$ of $g$ and the numerical solution, and the projection error, respectively.
Here $\mathcal{P}^*$ could be $\mathcal{P}$ or $\mathcal{P}^\pm$, and may vary for different $g$. For our unknown variables, the following
projections are taken as
\begin{eqnarray*}
\mathcal{P}^+\eta, \qquad \mathcal{P}\theta, \qquad \mathcal{P}\zeta,
\qquad \mathcal{P}w, \qquad \mathcal{P}^+u, \qquad \mathcal{P}p, \qquad \mathcal{P}q,  \qquad \mathcal{P}v.
\end{eqnarray*}
We would like to comment that these projections are for the purpose of proving the error estimate only,
and won't be used in our numerical method \eqref{scheme} nor its numerical implementation. 

Clearly the exact solutions also satisfy the equations (\ref{scheme}), and one can show that the following error equations
\begin{align}
  &
  \int_{I_j} ((\xi^\eta-\epsilon^\eta)-b(\xi^\theta-\epsilon^\theta))_t \rho dx
 -a\mathcal{A}_j(\xi^q-\epsilon^q,\rho;(\xi^q-\epsilon^q)^-)	\label{s1}\\
 & \quad
 =\int_{I_j} ((\xi^{u}-\epsilon^{u})+(u\eta-u_h\eta_h)) \rho_x dx - 
 ((u+\eta u-\widehat{F}_1)\rho^-)_{j+\frac{1}{2}}+((u+\eta u-\widehat{F}_1)\rho^+)_{j-\frac{1}{2}}, \notag\\ 
  &
  \int_{I_j} ((\xi^u-\epsilon^u)-d(\xi^p-\epsilon^p))_t \widetilde{\rho} dx
  -c\mathcal{A}_j(\xi^\zeta-\epsilon^\zeta,\widetilde{\rho};(\xi^\zeta-\epsilon^\zeta)^-) \label{s2}  \\
 & \quad
=\int_{I_j} ((\xi^{\eta}-\epsilon^{\eta})+f(u)-f(u_h)) \widetilde{\rho}_x dx
  -((\eta+f(u)-\widehat{F}_2)\widetilde{\rho}^-)_{j+\frac{1}{2}}
  +((\eta+f(u)-\widehat{F}_2)\widetilde{\rho}^+)_{j-\frac{1}{2}},  \notag\\
  &
  \int_{I_j} (\xi^v-\epsilon^v) \phi dx +\mathcal{A}_j(\xi^u-\epsilon^u,\phi;(\xi^u-\epsilon^u)^+)=0, \label{s3}\\
  &
  \int_{I_j} (\xi^p-\epsilon^p)\psi dx +\mathcal{A}_j(\xi^v-\epsilon^v,\psi;(\xi^v-\epsilon^v)^-)=0,\label{s4}\\
  &
  \int_{I_j} (\xi^q-\epsilon^q)\varphi dx+\mathcal{A}_j(\xi^v-\epsilon^v,\varphi;(\xi^v-\epsilon^v)^++\frac{\sign(c)}{2}[\xi^w-\epsilon^w])=0,\label{s5}\\
  &
  \int_{I_j} (\xi^w-\epsilon^w)\widetilde{\phi} dx+\mathcal{A}_j(\xi^\eta-\epsilon^\eta,\widetilde{\phi};(\xi^\eta-\epsilon^\eta)^+)=0,\label{s6}\\
  &
  \int_{I_j} (\xi^\theta-\epsilon^\theta) \widetilde{\psi} dx+\mathcal{A}_j(\xi^w-\epsilon^w,\widetilde{\psi};(\xi^w-\epsilon^w)^-)=0,\label{s7}\\
  &
  \int_{I_j} (\xi^\zeta-\epsilon^\zeta) \widetilde{\varphi} dx+ \mathcal{A}_j(\xi^w-\epsilon^w,\widetilde{\varphi};(\xi^w-\epsilon^w)^-+\frac{\sign(a)}{2}[\xi^v-\epsilon^v])=0,\label{s8}
\end{align}
hold for all the test functions $\rho, \phi,\psi,\widetilde{\rho}, \widetilde{\phi},$ $ \widetilde{\psi}, \varphi,\widetilde{\varphi}\in V^k_h$. 

Before stating the main results on error estimate, we would like to provide the following lemma that will be frequently used in the proof. 
This lemma, presented in \cite[Lemma 2.4]{XL}, provides the relation between the error of the variable $g_h$ and that of the other variable which approximates $g_x$.
  \begin{lemma}\label{xiaol}
    For $u_h,\eta_h,v_h,w_h$ defined in the DG method \eqref{scheme}, there exist the following inequalities 
    \begin{align*}
    &\norm{\xi^u_x}+h^{-\frac{1}{2}}|[\xi^u]|\le C\norm{\xi^v},\qquad
    \norm{\xi^\eta_x}+h^{-\frac{1}{2}}|[\xi^\eta]|\le C\norm{\xi^w}.
    \end{align*}
  \end{lemma}

To deal with the nonlinear term appearing in this model, we make the following a priori error estimate assumption
  \[\text{for}~ h ~\text{small enough}, ~~\norm{u-u_h}+\norm{\eta-\eta_h}\le h. \] 
The same technique has been studied in \cite{YS,XCY,XL} to treat the nonlinear term in the KdV, Keller-Segel and BBM equations.
We refer to these papers for the detailed explanation of this technique, and the verification to justify why such assumption could be made.

\subsection{The case of $a,~c<0~\text{and}~b,~d>0$}
\begin{theorem}\label{thm1}
  When $a,~c<0~\text{and}~b,~d>0$, let $u,\eta$ be the exact solutions to the system (\ref{1}) which are sufficiently smooth and bounded. Let $\eta_h,u_h\in V^k_h$ be 
  the numerical solutions of the LDG scheme (\ref{scheme}). For small enough $h$, there holds the following error estimate
  \begin{equation}\label{thm}
    \norm{u-u_h}^2+\norm{\eta-\eta_h}^2+\norm{v-v_h}^2+\norm{w-w_h}^2\le Ch^{2k+1}.
  \end{equation}
  \end{theorem}
  
Choose the test functions $\rho=-c\xi^\eta,~ \widetilde{\rho}=-a\xi^u$ in (\ref{s1}) and (\ref{s2}), 
$\phi=-ad\xi_t^v+ac\xi^\zeta, ~\widetilde{\phi}=-bc\xi_t^w+ac\xi^q$ in (\ref{s3}) and (\ref{s6}),
$\psi=-ad\xi^u, ~\widetilde{\psi}=-bc\xi^\eta$ after taking the time derivative of (\ref{s4}) and (\ref{s7}),
and $\varphi=-ac\xi^w, ~\widetilde{\varphi}=-ac\xi^v$ in (\ref{s5}) and (\ref{s8}), respectively.
Summing up all these equations and applying the property of the projection operator, i.e. 
\[\int_{I_j}\epsilon^p \phi dx=\int_{I_j}\epsilon^\theta \phi dx=\int_{I_j}\epsilon^v \phi dx
=\int_{I_j}\epsilon^w \phi dx=\int_{I_j}\epsilon_t^v \phi dx=\int_{I_j}\epsilon_t^w \phi dx=0, ~~~
\forall \phi\in V^k_h, \]
and 
\[\mathcal{A}_j(\epsilon^u, \phi; (\epsilon^u)^+)=\mathcal{A}_j(\epsilon^\eta, \phi; (\epsilon^\eta)^+)=0,\qquad
\forall \phi\in V^{k-1}_h, \]
we have
 \begin{equation}\label{eqq}
   \mathcal{P}_j-\mathcal{Q}_j=-c\mathcal{H}_j-a\mathcal{C}_j,
 \end{equation} 
where 
\begin{align}
&\mathcal{H}_j=\int_{I_j}((\xi^{u}-\epsilon^{u})+(u\eta-u_h\eta_h))\xi^{\eta}_x dx-((u+\eta 
  u-\widehat{F}_1)\xi^{\eta,-})_{j+\frac{1}{2}}+((u+\eta u-\widehat{F}_1)\xi^{\eta,+})_{j-\frac{1}{2}},	\label{def_hj}\\
&   \mathcal{C}_j=\int_{I_j}((\xi^{\eta}-\epsilon^{\eta})+f(u)-f(u_h))\xi^{u}_xdx-((\eta+f(u)-\widehat{F}_2)\xi^{u,-})_{j+\frac{1}{2}}
  +((\eta+f(u)-\widehat{F}_2)\xi^{u,+})_{j-\frac{1}{2}},	\label{def_cj}
\end{align}  
and 
\begin{eqnarray}\label{def_pj}
  \begin{split}
      \mathcal{P}_j=&\frac12\frac{\partial}{\partial t} \int_{I_j} 
  -c(\xi^{\eta})^2-bc(\xi^w)^2-a(\xi^{u})^2-ad(\xi^v)^2 dx
  -bc(\xi_t^{w,-}\xi^{\eta,+})_{j-\frac{1}{2}}+bc(\xi_t^{w,-}\xi^{\eta,+})_{j+\frac{1}{2}}\\
  &-ad(\xi_t^{v,-}\xi^{u,+})_{j-\frac{1}{2}}
  +ad(\xi_t^{v,-}\xi^{u,+})_{j+\frac{1}{2}} +ac(\xi^{q,-}\xi^{\eta,+})_{j-\frac{1}{2}}-ac(\xi^{q,-}\xi^{\eta,+})_{j+\frac{1}{2}}
 \\
  &+ac(\xi^{\zeta,-}\xi^{u,+})_{j-\frac{1}{2}}-ac(\xi^{\zeta,-}\xi^{u,+})_{j+\frac{1}{2}}
  -ac(\xi^{v,+}\xi^{w,-})_{j-\frac{1}{2}}+ac(\xi^{v,+}\xi^{w,-})_{j+\frac{1}{2}},
  \end{split}
\end{eqnarray}
\begin{eqnarray}\label{def_qj}
  \begin{split}
  \mathcal{Q}_j&=\int_{I_j} -c\xi^\eta \epsilon_t^\eta-a\xi^u \epsilon_t^u dx
  +ac\big((\epsilon^{q,-}\xi^{\eta,+})_{j-\frac{1}{2}}-(\epsilon^{q,-}\xi^{\eta,-})_{j+\frac{1}{2}}
  +(\epsilon^{\zeta,-}\xi^{u,+})_{j-\frac{1}{2}}-(\epsilon^{\zeta,-}\xi^{u,-})_{j+\frac{1}{2}} \big)
  \\
  &
  -ac\big((\epsilon^{v,+}-\frac{1}{2}[\epsilon^w]+\frac{1}{2}[\xi^w])\xi^{w,+}\big)_{j-\frac{1}{2}}
  +ac\big((\epsilon^{v,+}-\frac{1}{2}[\epsilon^w]+\frac{1}{2}[\xi^w])\xi^{w,-}\big)_{j+\frac{1}{2}}
  \\
  &-ad\big( (\epsilon_t^{v,-}\xi^{u,+})_{j-\frac{1}{2}}-(\epsilon_t^{v,-}\xi^{u,-})_{j+\frac{1}{2}}\big)
  -bc\big((\epsilon_t^{w,-}\xi^{\eta,+})_{j-\frac{1}{2}}-(\epsilon_t^{w,-}\xi^{\eta,-})_{j+\frac{1}{2}}\big)
  \\
  &
  -ac\big((\epsilon^{w,+}-\frac{1}{2}[\epsilon^v]+\frac{1}{2}[\xi^v])\xi^{v,+}\big)_{j-\frac{1}{2}}
  +ac\big((\epsilon^{w,+}-\frac{1}{2}[\epsilon^v]+\frac{1}{2}[\xi^v])\xi^{v,-}\big)_{j+\frac{1}{2}}.
  \end{split}  
\end{eqnarray} 
Summing over all the cells $I_j$, noting $|a|=-a,~|c|=-c$,  yields, 
\begin{eqnarray}\label{eqqq} 
   \sum^N_{j=1}\mathcal{P}_j-\sum^N_{j=1}\mathcal{Q}_j=\sum^N_{j=1}|c|\mathcal{H}_j
   +
  \sum^N_{j=1}|a|\mathcal{C}_j.
\end{eqnarray}

We would like to first establish Lemma \ref{lemma1} and Lemma \ref{lemma2}, and the proof of this theorem is 
a direct result of these two lemmas. Moreover, these lemmas will be used frequently in the proof of 
other theorems in this section.
  
\begin{lemma}\label{lemma1}
For the terms $\mathcal{P}_j$ and $\mathcal{Q}_j$ defined in \eqref{def_pj} and \eqref{def_qj}, we have 
  \begin{align}
    \sum^N_{j=1}\mathcal{P}_j &= \frac12\frac{\partial}{\partial t} \int_I |c|(\xi^{\eta})^2+|bc|(\xi^w)^2+|a|(\xi^{u})^2+|ad|(\xi^v)^2dx,  \\
    \sum^N_{j=1}\mathcal{Q}_j &\le \big(\norm{\xi^{u}}^2+\norm{\xi^{\eta}}^2 +\norm{\xi^w}^2+\norm{\xi^v}^2\big)+Ch^{2k+1}.
  \end{align}
\end{lemma}

\begin{proof}
By summing $\mathcal{P}_j$ in \eqref{def_pj} over all the element $I_j$, we have 
\[\sum^N_{j=1}\mathcal{P}_j=\frac12\frac{\partial}{\partial t}\int_I |c|(\xi^{\eta})^2+|bc|(\xi^w)^2+|a|(\xi^{u})^2+|ad|(\xi^v)^2 dx,\]
where all the cell interface terms cancelled due to the periodic boundary conditions. 
Similarly, summing $\mathcal{Q}_j$ in \eqref{def_qj} over all the elements $I_j$ yields
\begin{align} \label{eqn_qj}
  \sum^N_{j=1}\mathcal{Q}_j&=\int_I -c\xi^\eta \epsilon_t^\eta-a\xi^u \epsilon_t^u dx \\
  & \quad
	  +\sum^N_{j=1}\big((-bc\epsilon_t^{w,-}+ac\epsilon^{q,-})[\xi^\eta]\big)_{j+\frac{1}{2}}
	  -\sum^N_{j=1}ac\big((\epsilon^{v,+}-\frac{1}{2}[\epsilon^w]+\frac{1}{2}[\xi^w])[\xi^w]\big)_{j+\frac{1}{2}}  \notag\\
  & \quad
	  +\sum^N_{j=1}\big((-ad\epsilon_t^{v,-}+ac\epsilon^{\zeta,-})[\xi^u]\big)_{j+\frac{1}{2}}
	  -\sum^N_{j=1}ac\big((\epsilon^{w,-}-\frac{1}{2}[\epsilon^v]+\frac{1}{2}[\xi^v])[\xi^v]\big)_{j+\frac{1}{2}} . \notag
\end{align} 
By Young's inequality and the projection property (\ref{interpolation}), we can easily show that 
\[ 
\int_I -c\xi^\eta \epsilon_t^\eta-a\xi^u \epsilon_t^u dx \le \norm{\xi^u}^2+\norm{\xi^\eta}^2+Ch^{2k+2}.
\]
According to  the Young's inequality, the projection property, and Lemma 
\ref{xiaol}, we have
\begin{equation}\label{Yxiaol}
  \sum^N_{j=1}\big((-bc\epsilon_t^{w,-}+ac\epsilon^{q,-})[\xi^\eta]\big)_{j+\frac{1}{2}}
\le C\sum^N_{j=1}h^{-1}[\xi^\eta]^2_{j+\frac{1}{2}} + Ch\left(\norm{\epsilon_t^w}^2_{\Gamma_h}+\norm{\epsilon_t^q}^2_{\Gamma_h}\right)
\le C\norm{\xi^w}^2+Ch^{2k+2}.
\end{equation}

Similarly, we have
\[\sum^N_{j=1}\big((-ad\epsilon_t^{v,-}+ac\epsilon^{\zeta,-})[\xi^u]\big)_{j+\frac{1}{2}}
\le C\norm{\xi^v}^2+Ch^{2k+2}.\]
The Young's inequality and the projection property (\ref{interpolation}) lead to
\[\big|\sum^N_{j=1}\big(\epsilon^{v,+}[\xi^w]\big)_{j+\frac{1}{2}}\Big|\le 
\sum^N_{j=1}\frac{1}{4}[\xi^w]^2_{j+\frac{1}{2}}+Ch^{2k+1}, \qquad
\Big|\sum^N_{j=1}\big([\epsilon^w][\xi^w]\big)_{j+\frac{1}{2}}\Big|\le
\sum^N_{j=1}\frac{1}{2}[\xi^w]^2_{j+\frac{1}{2}}+Ch^{2k+1}.\]
Therefore,  
\[-\sum^N_{j=1}\big((\epsilon^{v,+}-\frac{1}{2}[\epsilon^w]+\frac{1}{2}[\xi^w])[\xi^w]\big)_{j+\frac{1}{2}} \le Ch^{2k+1},\]
and similarly,
\[-\sum^N_{j=1}\big((\epsilon^{w,-}-\frac{1}{2}[\epsilon^v]+\frac{1}{2}[\xi^v])[\xi^v]\big)_{j+\frac{1}{2}} \le Ch^{2k+1}.\]
Collecting all these approximations, Equation \eqref{eqn_qj} becomes
\begin{equation*}
 \sum^N_{j=1}\mathcal{Q}_j\le   (\norm{\xi^u}^2+\norm{\xi^\eta}^2+\norm{\xi^v}^2+\norm{\xi^w}^2)+Ch^{2k+1},
\end{equation*}
which finishes the proof.
\end{proof}

\begin{lemma}\label{lemma2}
For the approximation of the nonlinear terms $\mathcal{H}_j$ and $\mathcal{C}_j$ defined in \eqref{def_hj} and \eqref{def_cj}, we have
\begin{align} \label{eq_lemma2_1}
  \sum^N_{j=1}|c|\mathcal{H}_j\le 
  C\big(\norm{\xi^u}^2+\norm{\xi^\eta}^2+\norm{\xi^w}^2\big)+Ch^{2k+2},\\
  \sum^N_{j=1}|a|\mathcal{C}_j\le C\big(\norm{\xi^u}^2+\norm{\xi^\eta}^2+\norm{\xi^v}^2\big)+Ch^{2k+2}. \label{eq_lemma2_2}
\end{align}
\end{lemma}
\begin{proof}
Following the definition \eqref{def_hj} of $\mathcal{H}_j$, we have
  \[\sum^N_{j=1}\mathcal{H}_j=\int_I((\xi^u-\epsilon^u)+u\eta-u_h\eta_h)\xi_x^\eta dx
  +\sum^N_{j=1}\big((u+u\eta-\widehat{F}_1)[\xi^\eta]\big)_{j+\frac{1}{2}}=:\text{\RNum{1}}+\text{\RNum{2}},\]
which is separated into two terms to be treated differently.
We start with providing the bound of $\RNum{1}$. It can be observed that
  \[\int_I(\xi^u-\epsilon^u)\xi_x^\eta dx=\int_I\xi^u\xi_x^\eta dx\le\frac{1}{2}(\norm{\xi^u}^2+\norm{\xi_x^\eta}^2)
  \le C(\norm{\xi^u}^2+\norm{\xi^w}^2),\]
where the first equality follows from the projection property and the last inequality uses Lemma \ref{xiaol}.
Furthermore, the difference of the nonlinear term becomes
  \begin{eqnarray}\label{taylor}
    u\eta-u_h\eta_h=u(\xi^\eta-\epsilon^\eta)+\eta(\xi^u-\epsilon^u)-(\xi^\eta-\epsilon^\eta)(\xi^u-\epsilon^u).
  \end{eqnarray}
By applying Lemma \ref{xiaol} and the projection error \eqref{interpolation}, 
and utilizing the assumption that $u$ and $\eta$ are uniformly bounded, we have
  \begin{align*}
&    \int_I u(\xi^\eta-\epsilon^\eta)\xi_x^\eta dx
  \le C(\norm{\xi^\eta}^2+\norm{\epsilon^\eta}^2+\norm{\xi_x^\eta}^2)\le C(\norm{\xi^\eta}^2+\norm{\xi^w}^2)
  +Ch^{2k+2},  \\
&      \int_I \eta(\xi^u-\epsilon^u)\xi_x^\eta dx\le C(\norm{\xi^u}^2+\norm{\epsilon^u}^2+\norm{\xi_x^\eta}^2)\le C(\norm{\xi^u}^2+\norm{\xi^w}^2)+Ch^{2k+2},
  \end{align*} 
and, since $\xi^u-\epsilon^u=u-u_h$ is also uniformly bounded by a priori assumption, 
  \begin{equation*}
    \int_I (\xi^u-\epsilon^u)(\xi^\eta-\epsilon^\eta)\xi_x^\eta dx 
  \le C(\norm{\xi^u}^2+\norm{\xi^\eta}^2+\norm{\xi^w}^2)+Ch^{2k+2},
  \end{equation*}
which leads to the estimate 
\[\text{\RNum{1}}\le C(\norm{\xi^u}^2+\norm{\xi^\eta}^2+\norm{\xi^w}^2)+Ch^{2k+2}.\]

Next, we approximate the term $\text{\RNum{2}}$, which is separated into the sum of the following four terms
  \begin{eqnarray*}
      \begin{split}
        &\text{\RNum{2}}=\sum^N_{j=1}\Big((u+\eta u-\widehat{F}_1)[\xi^\eta]\Big)_{j+\frac{1}{2}}
      =\sum^N_{j=1}\Big((u+u\eta-\{u_h\}-\{\eta_h u_h\}+\frac{\alpha}{2}[\eta_h])
      [\xi^\eta]\Big)_{j+\frac{1}{2}}&\\
      &=\sum^N_{j=1}\Big((u-\{u_h\})[\xi^\eta]\Big)_{j+\frac{1}{2}}
      +\frac{1}{2}\sum^N_{j=1}\Big((\eta u-\eta_h^+u_h^+)[\xi^\eta]\Big)_{j+\frac{1}{2}}
      +\frac{1}{2} \sum^N_{j=1}\Big((\eta u-\eta_h^-u_h^-)[\xi^\eta]\Big)_{j+\frac{1}{2}}
      +\frac{\alpha}{2}\sum^N_{j=1}([\eta_h][\xi^\eta])_{j+\frac{1}{2}}&\\
      &=:\text{\RNum{2}}_1+\text{\RNum{2}}_2+\text{\RNum{2}}_3+\text{\RNum{2}}_4.
      \end{split}
    \end{eqnarray*}    
By Young's inequality and Lemma \ref{xiaol}, we have 
    \begin{eqnarray}\label{35}
      \begin{split}
        &\text{\RNum{2}}_1=\sum^N_{j=1}\Big((\{\xi^u\}-\{\epsilon^u\})[\xi^\eta]\Big)_{j+\frac{1}{2}}
        =\sum^N_{j=1}\Big(h^{\frac{1}{2}}(\{\xi^u-\epsilon^u\})h^{-\frac{1}{2}}
        [\xi^\eta]\Big)_{j+\frac{1}{2}}&\\
        &\le Ch\big(\norm{\xi^u}^2_{\Gamma_h}+\norm{\epsilon^u}^2_{\Gamma_h}\big)+C\sum^N_{j=1}h^{-1}[\xi^\eta]_{j+\frac{1}{2}}^2
       \le C\big(\norm{\xi^u}^2+ h^{2k+2}\big)+C\norm{\xi^w}^2.
      \end{split}
      \end{eqnarray}
Utilizing the decomposition (\ref{taylor}) and applying Young's inequality and Lemma \ref{xiaol} yield 
\begin{eqnarray}\label{36}
        \begin{split}
          &\text{\RNum{2}}_3=\frac{1}{2}\sum^N_{j=1}\Big(
          \big(u(\xi^{\eta,-}-\epsilon^{\eta,-})
          +\eta(\xi^{u,-}-\epsilon^{u,-})-(\xi^{\eta,-}-\epsilon^{\eta,-})(\xi^{u,-}-\epsilon^{u,-})\big)[\xi^\eta]\Big)
          _{j+\frac{1}{2}}&\\
          &\le 
          C\big(\norm{\xi^u}^2+\norm{\xi^\eta}^2+h^{2k+2}\big)+C\norm{\xi^w}^2.
        \end{split}
      \end{eqnarray}
Following the similar analysis, and noting that $(\epsilon^{\eta,+})_{j+\frac12}=(\epsilon^{u,+})_{j+\frac12}=0$ as a result of these Radau projections, 
we can obtain
\[\text{\RNum{2}}_2\le C(\norm{\xi^u}^2+\norm{\xi^\eta}^2)+C\norm{\xi^w}^2,\]
and furthermore, since $-[\eta_h]=[\eta-\eta_h]=[\xi^\eta-\epsilon^\eta]$,
\[\text{\RNum{2}}_4=\frac{\alpha}{2}\sum^N_{j=1}\Big([\epsilon^\eta-\xi^\eta][\xi^\eta]\Big)_{j+\frac12}
=\frac{\alpha}{2}\sum^N_{j=1}\Big([\epsilon^\eta][\xi^\eta]\Big)_{j+\frac12}-\frac{\alpha}{2}\sum^N_{j=1}[\xi^\eta]^2_{j+\frac12}\le
C(\norm{\xi^w}^2+h^{2k+2})+Ch\norm{\xi^w}^2.\] 
The combination of these results leads to
\[\sum^N_{j=1}\mathcal{H}_j=\text{\RNum{1}}+\text{\RNum{2}}
      \le C\big(\norm{\xi^u}^2+\norm{\xi^\eta}^2+\norm{\xi^w}^2\big)+Ch^{2k+2}.\]
Following the similar line of the proof, we can obtain the approximation of $\mathcal{C}_j$
    \[\sum^N_{j=1}\mathcal{C}_j  \le  C(\norm{\xi^u}^2+\norm{\xi^\eta}^2+\norm{\xi^v}^2)+Ch^{2k+2},\]
which finishes the proof.     
    \end{proof}
    
Combining the result of  Lemma \ref{lemma1} and Lemma \ref{lemma2} with \eqref{eqqq}, 
we obtain
\[\frac{\partial}{\partial t}\int_I |c|(\xi^{\eta})^2+|bc|(\xi^w)^2+|a|(\xi^{u})^2+|ad|(\xi^v)^2 dx
  \le C\Big(\norm{\xi^u}^2+\norm{\xi^\eta}^2+\norm{\xi^v}^2+\norm{\xi^w}^2\Big)+ Ch^{2k+1}. \]
Applying the Gronwall's inequality, combining with Lemma \ref{initial} on the optimal errors of the initial conditions, 
and the optimal projection error of $\norm{\epsilon^{u,\eta,v,w}}$, we have the error estimate \eqref{thm}, 
 which finishes the proof of Theorem \ref{thm1}.
\begin{remark}
As observed in \eqref{ldg_method}, two auxiliary variables ($\theta$ and $\zeta$) are introduced to approximate $w_x$, 
and as a result, the $\eta_{xxt}$ and $\eta_{xxx}$ terms are rewritten as $\theta_t$ and $\zeta_x$, respectively.
We use the alternating flux when evaluating $\theta_h$, and the upwind flux when evaluating $\zeta_h$.
The extra numerical dissipation introduced by this upwind flux is important in bounding the term $\mathcal{Q}_j$,
otherwise, the approximation will reduce to the order of $h^{2k}$.
For the same reason, the auxiliary variables ($p$ and $q$) are introduced to approximate $v_x$.
 \end{remark}

  \subsection{The case of $a=c\ge0,~~b,~d~>0$} 
\begin{theorem}\label{thm2}
  When $a=c\ge0,~~b,~d~>0$, let $u,\eta$ be the exact solutions to the system (\ref{1}) 
  which are sufficiently smooth and bounded. Let $\eta_h,u_h\in V^k_h$ be the numerical solutions of the 
  LDG scheme. 
  For small enough $h$, there holds the following error estimate  
  \begin{equation}\label{thm21}
    \norm{u-u_h}^2+\norm{\eta-\eta_h}^2+\norm{v-v_h}^2+\norm{w-w_h}^2\le Ch^{2k+1}.
  \end{equation}
In particular, if $a=c=0$, we have
 \begin{equation}\label{thm22}
    \norm{u-u_h}^2+\norm{\eta-\eta_h}^2+\norm{v-v_h}^2+\norm{w-w_h}^2\le Ch^{2k+2}.
  \end{equation}
  \end{theorem}
  \begin{proof}
The proof of this theorem is similar to that of Theorem \ref{thm1}. Below we will sketch the proof and mainly present the different steps.

Choose the test functions $\rho=\xi^\eta,~ \widetilde{\rho}=\xi^u$ in (\ref{s1}) and (\ref{s2}), 
$\phi=d\xi_t^v-a\xi^\zeta, ~\widetilde{\phi}=b\xi_t^w-a\xi^q$ in (\ref{s3}) and (\ref{s6}),
$\psi=d\xi^u, ~\widetilde{\psi}=b\xi^\eta$ after taking the time derivative of (\ref{s4}) and (\ref{s7}),
and $\varphi=a\xi^w, ~\widetilde{\varphi}=a\xi^v$ in (\ref{s5}) and (\ref{s8}), respectively.
Summing up all these equations and applying the property of the projection operator, we obtain 
\[\sum^N_{j=1}\mathcal{P}_j-\sum^N_{j=1}\mathcal{Q}_j=\sum^N_{j=1}\mathcal{H}_j+\sum^N_{j=1}\mathcal{C}_j,\]
where the nonlinear term $\mathcal{H}_j$ and $\mathcal{C}_j$ are defined in \eqref{def_hj} and \eqref{def_cj}, and 
\begin{eqnarray*}
  \begin{split}
  &\hskip3cm
  \sum^N_{j=1}\mathcal{P}_j=\frac12\frac{\partial}{\partial t}\int_{I} (\xi^\eta)^2+b(\xi^w)^2+(\xi^u)^2+d(\xi^v)^2 dx, \\
  &\sum^N_{j=1}\mathcal{Q}_j=\int_I \xi^\eta \epsilon_t^\eta+\xi^u \epsilon_t^u 
  dx+\sum^N_{j=1}\big((b\epsilon_t^{w,-}-a\epsilon^{q,-})[\xi^\eta]\big)_{j+\frac{1}{2}}
  +\sum^N_{j=1}a\big((\epsilon^{v,+}+\frac{1}{2}[\epsilon^w]-\frac{1}{2}[\xi^w])[\xi^w]\big)_{j+\frac{1}{2}}
  \\
  &+\sum^N_{j=1}\big((d\epsilon_t^{v,-}-a\epsilon^{\zeta,-})[\xi^u]\big)_{j+\frac{1}{2}}
  +\sum^N_{j=1}a\big((\epsilon^{w,-}+\frac{1}{2}[\epsilon^v]-\frac{1}{2}[\xi^v])[\xi^v]\big)_{j+\frac{1}{2}}
  =:\text{\RNum{1}}
  +\text{\RNum{2}}+\text{\RNum{3}}+\text{\RNum{4}}+\text{\RNum{5}}.
  \end{split}
  \end{eqnarray*}
Following the exactly same analysis as that in Lemma \ref{lemma1} yields
  \[\text{\RNum{1}}+\text{\RNum{2}}+\text{\RNum{4}}\le C\Big(\norm{\xi^\eta}^2+\norm{\xi^u}^2+\norm{\xi^w}^2+\norm{\xi^v}^2+h^{2k+2}\Big),\]
   then 
  \begin{eqnarray*}
    \text{\RNum{3}}+\RNum{5}=|a|\Bigg(\sum^N_{j=1}
    \big((\epsilon^{v,+}+\frac{1}{2}[\epsilon^w]-\frac{1}{2}[\xi^w])[\xi^w]\big)_{j+\frac{1}{2}}+
    \sum^N_{j=1}\big((\epsilon^{w,-}+\frac{1}{2}[\epsilon^v]-\frac{1}{2}[\xi^v])[\xi^v]\big)_{j+\frac{1}{2}}\Bigg)\le
    C |a|h^{2k+1}.
  \end{eqnarray*}
  Therefore, 
\begin{align*}
    \sum^N_{j=1}\mathcal{P}_j-\sum^N_{j=1}\mathcal{Q}_j 	 \ge 
    &\,\frac12\frac{\partial}{\partial t}\int_{I}(\xi^\eta)^2+b(\xi^w)^2+(\xi^u)^2+d(\xi^v)^2 dx\\
    &
    -C\Big(\norm{\xi^\eta}^2+\norm{\xi^u}^2+\norm{\xi^w}^2+\norm{\xi^v}^2+h^{2k+2}+|a|h^{2k+1}\Big).
\end{align*}
Combining this with the result in Lemma \ref{lemma2}, we have 
 \begin{eqnarray*}
     \frac{\partial}{\partial t}\int_{I}(\xi^\eta)^2+b(\xi^w)^2+(\xi^u)^2+d(\xi^v)^2 dx \le      
     C\Big(\norm{\xi^u}^2+\norm{\xi^w}^2+\norm{\xi^\eta}^2+\norm{\xi^v}^2\Big)
     +C\Big(h^{2k+2}+|a|h^{2k+1}\Big),
 \end{eqnarray*}
which leads to \eqref{thm21} when $a\neq 0$, and \eqref{thm22} when $a=0$.
  \end{proof}
\begin{remark}
We would like to point out that when $a=c=0$, this model reduces to the coupled BBM equations.
In \cite{JX}, LDG methods were presented for this model, based on slightly different reformulation of the original model. 
Both energy conserving and energy dissipative methods were studied in that paper, however, optimal error estimate analysis
was obtained only for the linearized model. In this theorem, we are able to provide the optimal error estimate for the fully nonlinear 
BBM-BBM equations.
\end{remark}  
    
\subsection{The case of either $b>0,~d=0$ or $b=0,~d>0$, and either $a,~c<0$ or $a=c\ge0$} 
Two theorems will be presented in this subsection, with the first one on the case of $d=0,~b>0$ and the other on the case of $b=0,~d>0$.  
  \begin{theorem}{\label{thm3}}
  When $b>0,~d=0$, either $a,~c<0$ or $a=c\ge0$, let $u,\eta$ be the exact solutions to the system (\ref{1}) which are sufficiently smooth and bounded. Let $\eta_h,u_h\in V^k_h$ be 
  the numerical solutions of the LDG scheme (\ref{scheme}). For small enough $h$, there holds the following error estimate
 \begin{equation}\label{eq_thm3}
    \norm{u-u_h}^2+\norm{\eta-\eta_h}^2+\norm{w-w_h}^2\le Ch^{2k+1}.
  \end{equation}
  \end{theorem}
Below we only present the proof for the case of $a,~c<0$, 
which follows the steps as that in Theorem \ref{thm1}. The other case of $a=c\ge0$ 
can be shown in a similar way, by combining the proof below with that of Theorem \ref{thm2}.  
\begin{proof}
   When $d=0, ~a,~c<0,~b>0$, the system (\ref{1}) reduces to
    \begin{equation}\label{1nn}
  \begin{cases}
\eta_t+u_x+(\eta u)_x+au_{xxx}-b\eta_{xxt}=0, &\\
u_t+\eta_x+\big(\frac{u^2}{2}\big)_x+c\eta_{xxx}=0.&
\end{cases}
\end{equation}
Follow the step as in the proof of Theorem \ref{thm1} to obtain
\begin{eqnarray}\label{eqqq1}   
   \sum^N_{j=1}\mathcal{P}_j-\sum^N_{j=1}\mathcal{Q}_j=\sum^N_{j=1}|c|\mathcal{H}_j+
  \sum^N_{j=1}|a|\mathcal{C}_j,
\end{eqnarray}
where $\mathcal{P}_j$, $\mathcal{Q}_j$, $\mathcal{H}_j$ and $\mathcal{C}_j$ have the exact same definitions, except with $d=0$.
From Lemma \ref{lemma1}, we have
\begin{align}\label{def_pj_2}
  &\sum^N_{j=1}\mathcal{P}_j=\frac12\frac{\partial}{\partial t} \int_I \Big( |c|(\xi^\eta)^2
  +|bc|(\xi^w)^2+|a|(\xi^u)^2\Big) dx,  \\ 
  &\sum^N_{j=1}\mathcal{Q}_j\le C\Big(\norm{\xi^u}^2+\norm{\xi^\eta}^2
  +\norm{\xi^w}^2+h^{2k+1}\Big)+ac\sum^N_{j=1}\Big(\epsilon^{\zeta,-}	 \label{def_qj_2}
  [\xi^u]\Big)_{j+\frac{1}{2}},
\end{align}
where the last term in $\sum^N_{j=1}\mathcal{Q}_j$ cannot be bounded as what is done in Lemma \ref{lemma1} due to $d=0$. 
We refer to \eqref{eq_lemma2_1} for the estimate of $\mathcal{H}_j$. 

The approximation of $\mathcal{C}_j$ needs to be handled differently,
since we can not allow $\norm{\xi^v}$, which does not appear in \eqref{def_pj_2}, to show up here. Recall that 
  \begin{eqnarray*}
          \sum^N_{j=1}|a|\mathcal{C}_j=\int_{I}|a|\big(\xi^\eta-\epsilon^\eta+\frac{u^2}{2}
          -\frac{u^2_h}{2}\big)\xi_x^udx+\sum^N_{j=1}|a|\Big(\big(\eta+\frac{u^2}{2}-\widehat{F}_2\big)[\xi^u]\Big)
          _{j+\frac{1}{2}},
  \end{eqnarray*}
which can be separated into the linear term and nonlinear term
  \begin{equation}\label{def_cj_2}
          \sum^N_{j=1}|a|\mathcal{C}_j=\int_{I}|a|\big(\xi^\eta-\epsilon^\eta\big)\xi_x^udx
          +\sum^N_{j=1}|a|\Big(\big(\eta-\{\eta_h\}\big)[\xi^u]\Big)_{j+\frac{1}{2}}
          +\int_{I}|a|\Big(\frac{u^2}{2}-\frac{u^2_h}{2}\Big)\xi_x^udx
          +\sum^N_{j=1}|a|\Big(\big(\frac{u^2}{2}-\widehat{f}\big)[\xi^u]\Big)_{j+\frac{1}{2}},
  \end{equation}
with $\widehat{f}=\widehat{F}_2-\{\eta_h\}=\frac{\{u_h^2\}}{2}-\frac{\alpha}{2}[u_h]$.
Applying the projection property and integration by parts, the linear part can be reformulated as
\begin{align}\label{firstterm}
  &\int_{I}|a|\big(\xi^\eta-\epsilon^\eta\big)\xi_x^udx+\sum^N_{j=1}|a|\Big(\big(\eta-\{\eta_h\}\big)[\xi^u]\Big)_{j+\frac{1}{2}} 
  =\int_{I}|a|\xi^\eta\xi_x^udx+\sum^N_{j=1}|a|\Big(\{\xi^\eta-\epsilon^\eta\}[\xi^u]\Big)_{j+\frac{1}{2}}\notag \\
  &\hskip2cm
  =-\int_{I}|a|\xi^\eta_x\xi^udx-\sum^N_{j=1}|a|\Big(\{\xi^\eta\}[\xi^u]+[\xi^\eta]\{\xi^u\}\Big)_{j+\frac{1}{2}}+\sum^N_{j=1}|a|
  \Big(\{\xi^\eta-\epsilon^\eta\}[\xi^u]\Big)_{j+\frac{1}{2}}\notag \\
  &\hskip2cm
  =-\int_{I}|a|\xi^\eta_x\xi^udx-\sum^N_{j=1}|a|\Big(\{\epsilon^\eta\}[\xi^u]+[\xi^\eta]\{\xi^u\}\Big)_{j+\frac{1}{2}}.
\end{align}
Utilizing Lemma \ref{xiaol}, the first term can be bounded as
$$
-\int_{I}|a|\xi^\eta_x\xi^udx \le |a| \norm{\xi^\eta_x} \norm{\xi^u} \le C \norm{\xi^w}^2+ \norm{\xi^u}^2,
$$
and the second term can be approximated by
\begin{align}\label{secondterm}
&-\sum^N_{j=1}|a|\Big(\{\epsilon^\eta\}[\xi^u]+[\xi^\eta]\{\xi^u\} \Big)_{j+\frac{1}{2}}
        =-\sum^N_{j=1}\Big(|a|h^{-\frac12}[\xi^\eta]h^\frac12\{\xi^u\} \Big)_{j+\frac{1}{2}}-
        \sum^N_{j=1}\Big(|a|\{\epsilon^\eta\}[\xi^u]\Big)_{j+\frac{1}{2}}  \\
       & \le Ch^{-1}\sum^N_{j=1}[\xi^\eta]^2_{j+\frac{1}{2}}+Ch\sum^N_{j=1}[\xi^u]^2_{j+\frac{1}{2}}-
        \sum^N_{j=1}\Big(|a|\{\epsilon^\eta\}[\xi^u]\Big)_{j+\frac{1}{2}} 
        \leq C\big(\norm{\xi^w}^2+\norm{\xi^u}^2\big)-\sum^N_{j=1}\Big(|a|\{\epsilon^\eta\} [\xi^u]\Big)_{j+\frac{1}{2}}.   \notag
\end{align}
To estimate the nonlinear terms in \eqref{def_cj_2}, we introduce the following lemma, which is exactly Lemma 3.4 and Lemma 3.5 in \cite{YS}, 
      and we will omit the proof here.
      \begin{lemma}\label{yanx}
       With $f(u)$ and $\widehat{f}$ defined above, there exits some positive number $\alpha(\widehat{f};u_h)$, such that 
        \[\int_I |a|\big(f(u)-f(u_h)\big)\xi_x^udx+
          \sum^N_{j=1}|a|\big((f(u)-\widehat{f})[\xi^u]\big)_{j+\frac{1}{2}}\le -\frac{1}{4}\sum^N_{j=1}\Big(\alpha(\widehat{f};u_h)
      [\xi^u]^2\Big)_{j+\frac{1}{2}}+C\big(\norm{\xi^u}^2+h^{2k+1}\big).\]
      \end{lemma} 
Combining all these estimates together, we have the following result
      \[ \sum^N_{j=1}|a|\mathcal{C}_j
      \le C \big(\norm{\xi^u}^2+\norm{\xi^w}^2+h^{2k+1}\big) 
      -\frac{1}{4}\sum^N_{j=1}\Big(\alpha(\widehat{f};u_h)[\xi^u]^2\Big)_{j+\frac{1}{2}}
      -\sum^N_{j=1}\Big(|a|\{\epsilon^\eta\} [\xi^u]\Big)_{j+\frac{1}{2}}.\]
Recall that 
        \[\sum^N_{j=1}\mathcal{P}_j-\sum^N_{j=1}\mathcal{Q}_j=\sum^N_{j=1}|c|\mathcal{H}_j+\sum^N_{j=1}|a|\mathcal{C}_j,\]
we have
\begin{align*}
     & \frac12\frac{\partial}{\partial t}\int_I \big(|c|(\xi^\eta)^2+|bc|(\xi^w)^2+|a|(\xi^u)^2\big)dx+\frac{1}{4}
     \sum^N_{j=1}\Big(\alpha(\widehat{f};u_h) [\xi^u]^2\Big)_{j+\frac{1}{2}} \\
     &\qquad
     \le C\big(\norm{\xi^u}^2+\norm{\xi^\eta}^2+\norm{\xi^w}^2+h^{2k+1}\big)   
          +ac\sum^N_{j=1}\Big(\epsilon^{\zeta,-}[\xi^u]\Big)_{j+\frac{1}{2}}
    	 -|a|\sum^N_{j=1}\Big(\{\epsilon^\eta\}[\xi^u]\Big)_{j+\frac{1}{2}},
\end{align*}
which leads to
\begin{align*}
 \frac{\partial}{\partial t}\int_I \big(|c|(\xi^\eta)^2+|a|(\xi^u)^2+|bc|(\xi^w)^2\big)dx\le C\big(\norm{\xi^\eta}^2+\norm{\xi^u}^2+\norm{\xi^w}^2\big)+ Ch^{2k+1},
\end{align*}
by applying the Young's inequality and the projection error estimate \eqref{interpolation}. 
Then the estimate \eqref{eq_thm3} follows from this, the Gronwall's inequality, Lemma \eqref{initial} on optimal initial conditions, and the optimal projection error \eqref{interpolation}.
\end{proof} 

\begin{theorem}{\label{thm3.5}}
  When  $b=0,~d>0$, either $a,~c<0$ or $a=c\ge0$, let $u,\eta$ be the exact solutions to the system (\ref{1}) which 
  are sufficiently smooth and bounded. Let $\eta_h,u_h\in V^k_h$ be 
  the numerical solutions of the LDG scheme (\ref{scheme}). For small enough $h$, there holds the following error estimate
 \begin{equation}\label{eq_thm3.5}
    \norm{u-u_h}^2+\norm{\eta-\eta_h}^2+\norm{v-v_h}^2\le Ch^{2k+1}.
  \end{equation}
  \end{theorem}
  
\begin{proof}
For the same reason, we only present the proof for the case of $a,~c<0$. 
Follow the steps in Theorem \ref{thm1}, we can derive
\begin{eqnarray*}
   \sum^N_{j=1}\mathcal{P}_j-\sum^N_{j=1}\mathcal{Q}_j=\sum^N_{j=1}|c|\mathcal{H}_j+
  \sum^N_{j=1}|a|\mathcal{C}_j,
\end{eqnarray*}
with
\begin{align}\label{def_pj_3}
  &\sum^N_{j=1}\mathcal{P}_j=\frac12\frac{\partial}{\partial t} \int_I \Big( |c|(\xi^\eta)^2
  +|ad|(\xi^v)^2+|a|(\xi^u)^2\Big) dx,  \\ 
  &\sum^N_{j=1}\mathcal{Q}_j\le C\Big(\norm{\xi^u}^2+\norm{\xi^\eta}^2
  +\norm{\xi^v}^2+h^{2k+1}\Big)+ac\sum^N_{j=1}\Big(\epsilon^{q,-} [\xi^\eta]\Big)_{j+\frac{1}{2}},
\end{align}
following the result in Lemma \ref{lemma1}, where the last term in $\sum^N_{j=1}\mathcal{Q}_j$ is different 
from that in the proof of Theorem \ref{thm3} due to $b=0$.  
Another difference in the proof is that the term $\mathcal{C}_j$ can now be bounded as in \eqref{eq_lemma2_2}, and 
the approximation of $\mathcal{H}_j$ needs to be handled differently, since we can not allow 
$\norm{\xi^w}$ that does not appear in \eqref{def_pj_3} to show up here.
Recall that 
\begin{eqnarray}\label{def_cj_3}
   \sum^N_{j=1}\mathcal{H}_j&=&\int_I(\xi^u-\epsilon^u)\xi^\eta_xdx+\sum^N_{j=1}\Big(\big(u-\{u_h\}\big)[\xi^\eta]
   \Big)_{j+\frac{1}{2}} \notag 
   \\&+&\int_I(u\eta-u_h\eta_h)\xi^\eta_xdx+  
   \sum^N_{j=1}\Big(\big(u\eta-\{u_h\eta_h\}+\frac{\alpha}{2}[\eta_h]\big)[\xi^\eta]\Big)_{j+\frac{1}{2}}.
\end{eqnarray}
For the linear term, we can follow the same analysis as in (\ref{firstterm}) and (\ref{secondterm}) to derive
\[\int_I(\xi^u-\epsilon^u)\xi^\eta_xdx+\sum^N_{j=1}\Big(\big(u-\{u_h\}\big)[\xi^\eta]
   \Big)_{j+\frac{1}{2}}\le C(\norm{\xi^\eta}^2+\norm{\xi^v}^2)+Ch^{2k+1}
   -\sum^N_{j=1}\Big(\{\epsilon^u\}[\xi^\eta]\Big)_{j+\frac{1}{2}}.\]
For the nonlinear term, we decompose it as the sum of $\Pi_1,~\Pi_2,~\Pi_3$, outlined below.
\begin{align}\label{def_cj_4}
&
 \int_I(u\eta-u_h\eta_h)\xi^\eta_xdx+ \sum^N_{j=1}\Big(\big(u\eta-\{u_h\eta_h\}+\frac{\alpha}{2}[\eta_h]\big)[\xi^\eta]\Big)_{j+\frac{1}{2}} 	\notag \\ 
&\hskip15mm     
      = \int_Iu(\xi^\eta-\epsilon^\eta)\xi^\eta_xdx+\sum^N_{j=1}\Big(u(\{\xi^\eta\}-\{\epsilon^\eta\}) [\xi^\eta]\Big)_{j+\frac12} 	\notag \\ 
&\hskip18mm     
    +\int_I \eta_h(\xi^u-\epsilon^u)\xi^\eta_xdx+\sum^N_{j=1}\Big(\{\eta_h(\xi^u-\epsilon^u)\}[\xi^\eta]\Big)_{j+\frac12}	\notag \\ 
&\hskip18mm     
   + \sum^N_{j=1}\Big(\frac{\alpha}{2}[\eta_h][\xi^\eta]\Big)_{j+\frac{1}{2}}  \quad =: \Pi_1+\Pi_2+\Pi_3.
\end{align}
One can easily observe that 
   \begin{eqnarray*}
        \Pi_3   =\sum^N_{j=1}\Big(\frac{\alpha}{2}[\epsilon^\eta-\xi^\eta][\xi^\eta]\Big)_{j+\frac{1}{2}}
      =\sum^N_{j=1}\Big(\frac{\alpha}{2}[\epsilon^\eta][\xi^\eta]\Big)_{j+\frac{1}{2}}
         -\sum^N_{j=1}\Big(\frac{\alpha}{2}[\xi^\eta]^2\Big)_{j+\frac{1}{2}}
       \le	Ch^{2k+1}-\frac{7\alpha}{16}\sum^N_{j=1}\Big([\xi^\eta]^2\Big)_{j+\frac12}.
   \end{eqnarray*}   
To bound $\Pi_2$, we rewrite it into the following form by adding and subtracting the same term
\begin{eqnarray*}
  \Pi_2   &=&\sum^N_{j=1} \int_{I_j} \eta_h(x_j)(\xi^u-\epsilon^u)\xi^\eta_xdx+ 
  \sum^N_{j=1}\eta_h(x_j)\Big(\{\xi^u-\epsilon^u\}[\xi^\eta]\Big)_{j+\frac12}\\
  &+&\sum^N_{j=1} \int_{I_j} \big(\eta_h-\eta_h(x_j)\big)(\xi^u-\epsilon^u)\xi^\eta_xdx+ 
  \sum^N_{j=1}\Big(\big\{\big(\eta_h-\eta_h(x_j)\big)(\xi^u-\epsilon^u)\big\}
  [\xi^\eta]\Big)_{j+\frac12}.
\end{eqnarray*}
Applying integration by parts, Young's inequality, Lemma \ref{xiaol}, the definition of the projection and its optimal error estimate (\ref{interpolation}), we obtain
\begin{align*}
 & \sum^N_{j=1} \int_{I_j} \eta_h(x_j)(\xi^u-\epsilon^u)\xi^\eta_xdx+ 
  \sum^N_{j=1}\eta_h(x_j)\Big(\{\xi^u-\epsilon^u\}[\xi^\eta]\Big)_{j+\frac12}\\
  &\hskip1cm
  =\sum^N_{j=1} \int_{I_j} \eta_h(x_j)\xi^u \xi^\eta_xdx +\sum^N_{j=1}\eta_h(x_j)\Big(\{\xi^u\}[\xi^\eta]\Big)_{j+\frac12}
  -\sum^N_{j=1}\eta_h(x_j)\Big(\{\epsilon^u\}[\xi^\eta]\Big)_{j+\frac12}\\
  &\hskip1cm
  =-\sum^N_{j=1}\int_{I_j}\eta_h(x_j)\xi^\eta\xi^u_xdx-\sum^N_{j=1}\eta_h(x_j)
  \Big([\xi^u]\{\xi^\eta\}\Big)_{j+\frac12}-\sum^N_{j=1}\eta_h(x_j)\Big(\{\epsilon^u\}[\xi^\eta]\Big)_{j+\frac12}  \\
  &\hskip1cm
  \le C(\norm{\xi^\eta}^2+\norm{\xi^v}^2)+Ch^{2k+1}+\frac{\alpha}{16}\sum^N_{j=1}\Big([\xi^\eta]^2\Big)_{j+\frac12}.
\end{align*}
Noting the fact that $\eta_h(x)-\eta_h(x_j)=O(h)$ for any $x\in I_j$, we can apply 
Young's inequality, projection property (\ref{interpolation}), and the inverse inequality \eqref{inverse} to obtain
  \begin{align*}
    &\Big|\sum^N_{j=1} \int_{I_j} \big(\eta_h-\eta_h(x_j)\big)(\xi^u-\epsilon^u)\xi^\eta_xdx+ 
  \sum^N_{j=1}\Big(\big\{\big(\eta_h-\eta_h(x_j)\big)(\xi^u-\epsilon^u)\big\}
  [\xi^\eta]\Big)_{j+\frac12}\Big|\\
  &\le \sum^N_{j=1}\int_{I_j}Ch|(\xi^u-\epsilon^u)\xi^\eta_x|dx+\sum^N_{j=1}
  \Big|Ch\{\xi^u-\epsilon^u\}[\xi^\eta]\Big|_{j+\frac12} \le C(\norm{\xi^u}^2+\norm{\xi^\eta}^2)+Ch^{2k+2}.
  \end{align*}
Therefore we obtain the following estimate
  \[\Pi_2\le C(\norm{\xi^\eta}^2+\norm{\xi^v}^2+\norm{\xi^u}^2)+Ch^{2k+1}+\frac{\alpha}{16}\sum^N_{j=1}\Big([\xi^\eta]^2\Big)_{j+\frac12}.\]

  Next we will estimate the term $\Pi_1$. From integration by parts and the property of projection, we have
  \begin{eqnarray}\label{ibp}
    \sum^N_{j=1}\int_{I_j}u(x_j)(\xi^\eta-\epsilon^\eta)\xi^\eta_xdx+\sum^N_{j=1}
    \Big(u(x_j)\{\xi^\eta\}[\xi^\eta]\Big)_{j+\frac12}=0,
  \end{eqnarray}
therefore, we obtain
  \begin{eqnarray*}
    \Pi_1  &=&\sum^N_{j=1}\int_{I_j}u(\xi^\eta-\epsilon^\eta)\xi^\eta_xdx+\sum^N_{j=1}
    \Big(u(\{\xi^\eta\}-\{\epsilon^\eta\})
    [\xi^\eta]\Big)_{j+\frac12}\\
    &-&\sum^N_{j=1}\int_{I_j}u(x_j)(\xi^\eta-\epsilon^\eta)\xi^\eta_xdx-\sum^N_{j=1}
    \Big(u(x_j)\{\xi^\eta\}[\xi^\eta]\Big)_{j+\frac12}\\
    &=&\sum^N_{j=1}\int_{I_j}\big(u-u(x_j)\big)(\xi^\eta-\epsilon^\eta)\xi^\eta_xdx
    +\sum^N_{j=1}\Big(\big(u-u(x_j)\big)\{\xi^\eta\}[\xi^\eta]\Big)_{j+\frac12}
    -\sum^N_{j=1}\Big(u\{\epsilon^\eta\}[\xi^\eta]\Big).
  \end{eqnarray*}
  Since $u(x)-u(x_j)=O(h)$ for any $x\in I_j$, applying the Young's inequality, inverse inequality (\ref{inverse}), and projection property (\ref{interpolation}) yields
  \begin{align*}
    \sum^N_{j=1}\int_{I_j}(u-u(x_j))(\xi^\eta-\epsilon^\eta)\xi^\eta_xdx
    &\le \sum^N_{j=1}\int_{I_j}C|(\xi^\eta-\epsilon^\eta)h\xi^\eta_x|dx
    \le C(\norm{\xi^\eta}^2+h^{2k+2}), \\
    \sum^N_{j=1}\Big(\big(u-u(x_j)\big)\{\xi^\eta\}[\xi^\eta]\Big)_{j+\frac12}
    & \le \sum^N_{j=1}\Big(Ch|\{\xi^\eta\}[\xi^\eta]|\Big)_{j+\frac12}
    \le C\norm{\xi^\eta}^2, \\	
    -\sum^N_{j=1}\Big(u\{\epsilon^\eta\}[\xi^\eta]\Big)
    & \le  Ch^{2k+1}+\frac{\alpha}{8}\sum^N_{j=1}\Big([\xi^\eta]^2\Big)_{j+\frac12}.
  \end{align*}
  Therefore, we obtain the estimate of $\Pi_1$ as
  \[\Pi_1\le C(\norm{\xi^\eta}^2+h^{2k+1})+\frac{\alpha}{8}\sum^N_{j=1}\Big([\xi^\eta]^2\Big)_{j+\frac12}.\]
  and the estimate of the nonlinear term is given by
  \begin{eqnarray*}
    \int_I(u\eta-u_h\eta_h)\xi^\eta_xdx+  
   \sum^N_{j=1}\Big(\big(u\eta-\{u_h\eta_h\}+\frac{\alpha}{2}[\eta_h]\big)[\xi^\eta]^2\Big)_{j+\frac{1}{2}}
   =\Pi_1+\Pi_2+\Pi_3
   ~~~~~~~~~~~~~~~
   \\
   \le C((\norm{\xi^\eta}^2+\norm{\xi^v}^2+\norm{\xi^u}^2))+Ch^{2k+1}-\frac{\alpha}{4}\sum^N_{j=1}\Big([\xi^\eta]^2\Big)_{j+\frac12}
   .
  \end{eqnarray*}
  As a result, we have shown that
  \[\sum^N_{j=1}\mathcal{H}_j\le C(\norm{\xi^\eta}^2+\norm{\xi^v}^2+\norm{\xi^u}^2)+Ch^{2k+1}-\frac{\alpha}{4}\sum^N_{j=1}\Big([\xi^\eta]^2\Big)_{j+\frac12}
  -\sum^N_{j=1}\Big(\{\epsilon^u\}[\xi^\eta]\Big)_{j+\frac{1}{2}}.\]
  Recall the relation 
  \[\sum^N_{j=1}\mathcal{P}_j-\sum^N_{j=1}\mathcal{Q}_j=\sum^N_{j=1}|c|\mathcal{H}_j+\sum^N_{j=1}|a|\mathcal{C}_j,\]
  we have 
  \begin{align*}
    &\frac12\frac{\partial}{\partial t} \int_I \Big( |c|(\xi^\eta)^2
  +|ad|(\xi^v)^2+|a|(\xi^u)^2\Big) 
  dx+\frac{|c|\alpha}{4}\sum^N_{j=1}\Big([\xi^\eta]^2\Big)_{j+\frac12}\\
    &\qquad
  \le C(\norm{\xi^\eta}^2+\norm{\xi^v}^2+\norm{\xi^u}^2)+Ch^{2k+1}+
  ac\sum^N_{j=1}\Big(\epsilon^{q,-}	                               
  [\xi^\eta]\Big)_{j+\frac{1}{2}}-\sum^N_{j=1}\Big(|c|\{\epsilon^u\}[\xi^\eta]\Big)_{j+\frac{1}{2}},
  \end{align*}
  which leads to 
  \begin{eqnarray*}
    \frac{\partial}{\partial t} \int_I \Big( |c|(\xi^\eta)^2
  +|ad|(\xi^v)^2+|a|(\xi^u)^2\Big) 
  dx\le C(\norm{\xi^\eta}^2+\norm{\xi^v}^2+\norm{\xi^u}^2)+Ch^{2k+1}.
  \end{eqnarray*}
The estimate \eqref{eq_thm3.5} follows from the Gronwall's inequality, Lemma \eqref{initial}, and the optimal projection error \eqref{interpolation}.
\end{proof}

\subsection{The case of $b,~d>0$, $ac=0$ and $a,~c\leq 0$}
\begin{theorem}\label{thm4}
  When $b,~d>0$, $ac=0$ and $a,~c\leq 0$, let $u,\eta$ be the exact solutions to the system (\ref{1}) which are sufficiently smooth and bounded. Let $\eta_h,u_h\in V^k_h$ be the numerical solutions of the LDG scheme (\ref{scheme}). For small enough $h$, there holds the following error estimate
\begin{equation}\label{eq_thm4}
    \norm{u-u_h}^2+\norm{\eta-\eta_h}^2+\norm{v-v_h}^2+\norm{w-w_h}^2+\norm{p-p_h}^2\le Ch^{2k}.
  \end{equation}
\end{theorem}
Below we only prove the case $c=0$, and 
the other case of $a=0$ can be done following a similar idea. 
One more projection of $v$, defined as
 \begin{eqnarray*}
  \tilde{\xi}^v=\mathcal{P}^-v-v_h,~~\tilde{\epsilon}^v=\mathcal{P}^-v-v, \qquad \text{with  } \tilde{\xi}^v-\tilde{\epsilon}^v=v-v_h,
\end{eqnarray*}
is needed in the proof, and this leads to the following extra error equations, in addition to \eqref{s1}-\eqref{s8},
 \begin{align}
  &\int_{I_j}(\tilde{\xi}^v-\tilde{\epsilon^v})\phi^* dx+\mathcal{A}_j(\xi^u-\epsilon^u,\phi^*;(\xi^u-\epsilon^u)^+)=0, \label{s55}\\
  &\int_{I_j}(\xi^p-\epsilon^p)\psi^* dx+\mathcal{A}_j(\tilde{\xi}^v-\tilde{\epsilon}^v,\psi^*;(\tilde{\xi}^v-\tilde{\epsilon}^v)^-)=0.\label{s66}
\end{align}
 
\begin{proof} Different set of test functions is chosen here, since the term $c\eta_{xxx}$ disappears when $c=0$.
Choose the test functions $\rho=\xi^\eta,~ \widetilde{\rho}=\xi^u+a\xi^p$ in (\ref{s1}) and (\ref{s2}), 
$\phi=d\xi_t^v, ~\widetilde{\phi}=b\xi_t^w-a(\xi^q-\xi^p)$ in (\ref{s3}) and (\ref{s6}),
$\psi=-a\xi^w,~\varphi=a\xi^w,~\psi^*=-a\xi^u_t$ in (\ref{s4}), (\ref{s5}), and (\ref{s66}),
and $\psi=d\xi^u, ~\widetilde{\psi}=b\xi^\eta, ~\phi^*=-a\tilde{\xi^v}$ after taking the time derivative of (\ref{s4}), (\ref{s7}) and (\ref{s55}), respectively.
Summing up all the equations involved above, applying the projection properties, and noting that $c=0$, we have 
\begin{equation}\label{eqqq2}
  \mathcal{P}^*_j-\mathcal{Q}^*_j=\mathcal{H}_j+\mathcal{C}_j-\mathcal{D}_j,
\end{equation}
where 
\begin{eqnarray}\label{defpjj}
  \begin{split}
  \mathcal{P}^*_j=&\frac{\partial}{\partial t}\int_{I_j} 
(\xi^\eta)^2+(\xi^u)^2+b(\xi^w)^2+d
(\xi^v)^2-a(\tilde{\xi}^v)^2-ad(\xi^p)^2 dx\\
+&b\Big(\big(\xi_t^{w,-}\xi^{\eta,+}\big)_{j-\frac{1}{2}}-
\big(\xi_t^{w,-}\xi^{\eta,+}\big)_{j+\frac{1}{2}}\Big)
+d\Big(\big(\xi_t^{v,-}\xi^{u,+}\big)_{j-\frac{1}{2}}-
\big(\xi_t^{v,-}\xi^{u,+}\big)_{j+\frac{1}{2}}\Big)\\
-&a\Big(\big(\tilde{\xi}^{v,-}\xi_t^{u,+}\big)_{j-\frac{1}{2}}-
\big(\tilde{\xi}^{v,-}\xi_t^{u,+}\big)_{j+\frac{1}{2}}\Big)
-a\Big(\big(\xi^{q,-}\xi^{\eta,+}\big)_{j-\frac{1}{2}}-
\big(\xi^{q,-}\xi^{\eta,+}\big)_{j+\frac{1}{2}}\Big)\\
-&a\Big(\big(\xi^{v,-}\xi^{w,+}\big)_{j-\frac{1}{2}}-
\big(\xi^{v,-}\xi^{w,-}\big)_{j+\frac{1}{2}}\Big)
-a\Big(-\big(\xi^{v,+}\xi^{w,+}\big)_{j-\frac{1}{2}}+
\big(\xi^{v,+}\xi^{w,-}\big)_{j+\frac{1}{2}}\Big),
\end{split}
\end{eqnarray}
\begin{eqnarray}\label{defqjj}
  \begin{split}
      \mathcal{Q}_j^*=&\int_I\big((\epsilon^\eta_t\xi^\eta+\epsilon^u_t(\xi^u+a\xi^p)
-a\varepsilon^v_t\tilde{\xi}^v\big)dx
+b\Big(\big(\epsilon_t^{w,-}\xi^{\eta,+}\big)_{j-\frac{1}{2}}-
\big(\epsilon_t^{w,-}\xi^{\eta,-}\big)_{j+\frac{1}{2}}\Big)\\
+&d\Big(\big(\epsilon_t^{v,-}\xi^{u,+}\big)_{j-\frac{1}{2}}-
\big(\epsilon_t^{v,-}\xi^{u,-}\big)_{j+\frac{1}{2}}\Big)-a
\Big(\big(\epsilon_t^{u,+}\tilde{\xi}^{v,+}\big)_{j-\frac{1}{2}}-
\big(\epsilon_t^{u,+}\tilde{\xi}^{v,-}\big)_{j+\frac{1}{2}}\Big)\\
-&a\Big(\big(\epsilon^{v,-}\xi^{w,+}\big)_{j-\frac{1}{2}}-
\big(\epsilon^{v,-}\xi^{w,-}\big)_{j+\frac{1}{2}}\Big)
-a\Big(-\big(\epsilon^{v,+}\xi^{w,+}\big)_{j-\frac{1}{2}}+
\big(\epsilon^{v,+}\xi^{w,-}\big)_{j+\frac{1}{2}}\Big)\\
-&a\Big(\big(\epsilon^{q,-}\xi^{\eta,+}\big)_{j-\frac{1}{2}}-
\big(\epsilon^{q,-}\xi^{\eta,-}\big)_{j+\frac{1}{2}}\Big),
  \end{split}
\end{eqnarray}
$\mathcal{H}_j$ and $\mathcal{C}_j$ are defined in \eqref{def_hj}-\eqref{def_cj}, and 
\begin{eqnarray*}
    \mathcal{D}_j=\int_{I_j}(f(u)-f(u_h))|a|\xi^p_xdx-\big((\eta_h^+-\{\eta_h\}
    +f(u)-\widehat{f})|a|\xi^{p,-}\big)_{j+\frac{1}{2}} +\big((\eta^+_h-\{\eta_h\}
    +f(u)-\widehat{f})|a|\xi^{p,+}\big)_{j-\frac{1}{2}}.
  \end{eqnarray*}
Summing over  all the elements $I_j$, we obtain
\[\sum^N_{j=1}\mathcal{P}^*_j-\sum^N_{j=1}\mathcal{Q}^*_j=\sum^N_{j=1}\mathcal{H}_j+
\sum^N_{j=1}\mathcal{C}_j-\sum^N_{j=1}\mathcal{D}_j,\]
where 
\[\sum^N_{j=1}\mathcal{P}^*_j=\frac{\partial}{\partial t}\int_I 
(\xi^\eta)^2+(\xi^u)^2+b(\xi^w)^2+d
(\xi^v)^2+|a|(\tilde{\xi}^v)^2+|ad|(\xi^p)^2 dx-|a|\sum^N_{j=1}\big([\xi^v][\xi^w]\big)_{j+\frac{1}{2}},\]
and 
\begin{eqnarray*}
  \sum^N_{j=1}\mathcal{Q}^*_j=\int_I\big(\epsilon^\eta_t\xi^\eta+\epsilon^u_t(\xi^u+a\xi^p)
-a\varepsilon^v_t\tilde{\xi}^v\big)dx+\sum^N_{j=1}\big((-a\epsilon^{q,-}+b\epsilon_t^{w,-})
[\xi^\eta]\big)_{j+\frac{1}{2}}\\
+\sum^N_{j=1}\big(-a\epsilon_t^{u,+}[\tilde{\xi}^v]\big)_{j+\frac{1}{2}}
+\sum^N_{j=1}\big(d\epsilon_t^{v,-}[\xi^u]\big)_{j+\frac{1}{2}}+a
\sum^N_{j=1}\big([\epsilon^v][\xi^w]\big)_{j+\frac{1}{2}}.
\end{eqnarray*}

Recall Lemma \ref{lemma2}, we have 
\[\sum^N_{j=1}\Big(\mathcal{H}_j+\mathcal{C}_j\Big)\le C\big(\norm{\xi^{u}}^2+\norm{\xi^{\eta}}^2
 +\norm{\xi^w}^2+\norm{\xi^v}^2\big)+Ch^{2k+2}.\]
Applying the Young's inequality and the projection property (\ref{interpolation}) leads to 
\[\int_I\big(\epsilon^\eta_t\xi^\eta+\epsilon^u_t(\xi^u+a\xi^p)
-a\varepsilon^v_t\tilde{\xi}^v\big)dx\le C\big(
\norm{\xi^u}^2+\norm{\xi^\eta}^2+\norm{\xi^p}^2+\norm{\tilde{\xi}^v}^2\big)+Ch^{2k+2},\]
and utilizing Lemma \ref{xiaol} (and its analogue for $[\tilde{\xi}^v]$) yields
\[\sum^N_{j=1}\big((-a\epsilon^{q,-}+b\epsilon_t^{w,-})
[\xi^\eta]\big)_{j+\frac{1}{2}}+\sum^N_{j=1}(-a\epsilon_t^{u,+}[\tilde{\xi}^v])_{j+\frac{1}{2}}
+\sum^N_{j=1}(d\epsilon_t^{v,-}[\xi^u])_{j+\frac{1}{2}}
\le C\big(\norm{\xi^w}^2+\norm{\xi^p}^2+\norm{{\xi}^v}^2\big)+Ch^{2k+2}.\]
We can apply Young's inequality, the projection property (\ref{interpolation}) and the inverse inequality \eqref{inverse} to obtain
\[\sum^N_{j=1}\big([\epsilon^v][\xi^w]\big)_{j+\frac{1}{2}}
=\sum^N_{j=1}\big(h^{-\frac{1}{2}}[\epsilon^v]h^{\frac{1}{2}}[\xi^w]\big)_{j+\frac{1}{2}}\le C\norm{\xi^w}^2+Ch^{2k}.\]
Combining these results leads to
\[\sum^N_{j=1}\mathcal{Q}^*_j\le C\big(\norm{\xi^u}^2+\norm{\xi^\eta}^2+\norm{\xi^p}^2+\norm{\tilde{\xi}^v}^2
+\norm{\xi^w}^2+\norm{{\xi}^v}^2\big)+Ch^{2k}.\]

Next, we would like to estimate $\sum_{j=1}^N \mathcal{D}_j$, which takes the form of 
 \[\sum^N_{j=1}\mathcal{D}_j=|a|\int_I\big(f(u)-f(u_h)\big)\xi^p_x dx
 +|a|\sum^N_{j=1}\big((\eta_h^+-\{\eta_h\}+f(u)-\widehat{f})[\xi^p]\big)
 _{j+\frac{1}{2}}=:|a|\Lambda_1+|a|\Lambda_2.\]
Notice that
\[f(u)-f(u_h)=\frac{u^2}{2}-\frac{(u_h)^2}{2}=\frac{1}{2}(u+u_h)(u-u_h)=\frac{1}{2}(u+u_h)(\xi^u-\epsilon^u).\]
If we define $\mathcal{L}(x)=(u+u_h)/2$, then by mean value theorem we have $|\mathcal{L}(x)-\mathcal{L}(x_j)|=O(h)$ for $x\in I_j$. 
Applying integration by parts, together with Young's inequality, the inverse inequality \eqref{inverse}, Lemma \ref{xiaol} and the projection property (\ref{interpolation}), we have
 \begin{eqnarray*}
   \big|\Lambda_1\big|&=&\Big|\sum^N_{j=1}\int_{I_j}\big(f(u)-f(u_h)\big)\xi^p_xdx\Big|
   =\Big|\sum^N_{j=1}\int_{I_j}\mathcal{L}(x)(\xi^u-\epsilon^u)\xi^p_x dx\Big|\\
   &=&\Big|\sum^N_{j=1}\int_{I_j}\big(\mathcal{L}(x)-\mathcal{L}(x_j)\big)(\xi^u-\epsilon^u)
   \xi^p_x dx+\sum^N_{j=1}\int_{I_j}\mathcal{L}(x_j)(\xi^u-\epsilon^u)\xi_x^p dx\Big|
   \\
   &\le&\sum^N_{j=1}\int_{I_j}C|\xi^u-\epsilon^u|h|\xi^p_x| dx+\Big|
   -\sum^N_{j=1}\int_{I_j}\mathcal{L}(x_j)\xi^u_x\xi^pdx
   -\sum^N_{j=1}\mathcal{L}(x_j)[\xi^u\xi^p]_{j+\frac{1}{2}}\Big|\\
   &\le& C\big(\norm{\xi^u}^2+\norm{\xi^p}^2+h^{2k+2}\big)
   +C\big(\norm{\xi^v}^2+\norm{\xi^p}^2\big)+C\big(\norm{\xi^v}^2
   +\norm{\xi^p}^2\big).
 \end{eqnarray*}
We now turn to the approximation of $\Lambda_2$, which has the following explicit expression
 \begin{eqnarray*}
   \Lambda_2=\sum^N_{j=1}\big((\eta_h^+-\{\eta_h\}+f(u)-\widehat{f})[\xi^p]\big)
 _{j+\frac{1}{2}}=\sum^N_{j=1}\big((\eta_h^+-\{\eta_h\}+\frac{u^2}{2}-\frac{\{u_h\}^2}{2}
 +\frac{\alpha}{2}[u_h])[\xi^p]\big)_{j+\frac{1}{2}}.
 \end{eqnarray*}
Applying the Young's inequality, the projection property  (\ref{interpolation}), the inverse inequality \eqref{inverse} and Lemma \ref{xiaol}, we have
\begin{align*}
&\Big|\sum^N_{j=1}\big((\eta_h^+-\{\eta_h\}\big)[\xi^p]\big)_{j+\frac{1}{2}}\Big|
 =\Big|\frac{1}{2}\sum^N_{j=1}\big([\xi^\eta-\epsilon^\eta][\xi^p]\big)_{j+\frac{1}{2}}\Big|
 \le C\big(\norm{\xi^w}^2+\norm{\xi^p}^2+h^{2k}\big), \\
&\Big|\sum^N_{j=1}\frac{\alpha}{2}\big([u_h][\xi^p]\big)_{j+\frac{1}{2}}\Big|=\Big|
   \sum^N_{j=1}\frac{\alpha}{2}\big([u_h-u][\xi^p]\big)_{j+\frac{1}{2}}\Big|
   \le C\sum^N_{j=1}\Big|\big([\epsilon^u-\xi^u][\xi^p]\big)_{j+\frac{1}{2}}\Big|
   \le C\big(\norm{\xi^v}^2+\norm{\xi^p}^2+h^{2k}\big).
\end{align*}
Similarly, utilizing the fact that $u+\{u_h\}$ is uniformly bounded leads to
 \[\Big|\sum^N_{j=1}\big(\frac{u^2}{2}-\frac{\{u_h\}^2}{2}\big)[\xi^p]_{j+\frac{1}{2}}\Big|
 =\frac{1}{2}\Big|\sum^N_{j=1}\big((u+\{u_h\})(\{\xi^u\}-\{\epsilon^u\})[\xi^p]\big)_{j+\frac{1}{2}}\Big|
 \le C\big(\norm{\xi^v}^2+\norm{\xi^p}^2+h^{2k}\big).\]
 Therefore, 
 \[|\Lambda_2|\le C\big(\norm{\xi^v}^2+\norm{\xi^p}^2+\norm{\xi^w}^2\big)+Ch^{2k},\]
which leads to
 \[\Big|\sum^N_{j=1}\mathcal{D}_j\Big|=|a\Lambda_1+a\Lambda_2|\le |a|\big(|\Lambda_1|+|\Lambda_2|\big)\le C\big(\norm{\xi^u}^2+\norm{\xi^p}^2
 +\norm{\xi^w}^2+\norm{\xi^v}^2+h^{2k}\big).\]
 
 Finally recall the equality 
 \[\sum^N_{j=1}\mathcal{P}^*_j=\sum^N_{j=1}\big(\mathcal{Q}^*_j+\mathcal{C}_j+\mathcal{H}_j-\mathcal{D}_j\big),\]
 and by Lemma \ref{xiaol} and Young's inequality, 
 \[\sum^N_{j=1}\big([\xi^v][\xi^w]\big)_{j+\frac{1}{2}}\le C\big(\norm{\xi^p}^2+\norm{\xi^w}^2\big),\]
 we obtain 
 \[\frac{\partial}{\partial t}\int_I 
(\xi^\eta)^2+(\xi^u)^2+b(\xi^w)^2+d
(\xi^v)^2+|a|(\tilde{\xi}^v)^2+|ad|(\xi^p)^2 dx\le 
C\big(\norm{\xi^{u}}^2+\norm{\xi^{\eta}}^2
 +\norm{\xi^w}^2+\norm{\xi^v}^2+\norm{\tilde{\xi}^v}^2+\norm{\xi^p}^2+h^{2k}\big).\]
We can apply the Gronwall's inequality and 
the estimate \eqref{eq_thm4} follows from Lemma \ref{initial} and the optimal projection error \eqref{interpolation}. 
 \end{proof}

 \subsection{The case of $a=b=c=d=0$}
 \begin{theorem}
   When $a=b=c=d=0$, 
   let $u,\eta$ be the exact solutions to the system (\ref{1}) which are sufficiently smooth and bounded. 
   Let $\eta_h,u_h\in V^k_h$ be the numerical solutions of the LDG scheme (\ref{scheme}). 
   For small enough $h$, there holds the following error estimate
\begin{equation}
    \norm{u-u_h}^2+\norm{\eta-\eta_h}^2\le Ch^{2k+1}.
  \end{equation}
 \end{theorem}
Under this assumption, the abcd Boussinesq system reduces to the nonlinear shallow water equations, which is 
a first order symmetrizable system of conservation law. Therefore, the above theorem is a direct application of \cite[Theorem 2.1]{QZ}.

 \section{Numerical Experiments}\label{num} \setcounter{equation}{0} \setcounter{table}{0} \setcounter{figure}{0}
The numerical results of the proposed LDG methods are presented in this section. The third order strong stability preserving Runge-Kutta methods \cite{SSP} are used as the temporal discretization. We first perform the accuracy test for different choices of the parameters $a,b,c,d$, and then present several wave collision tests, including a simulation with finite-time blow-up wave, to demonstrate the performance of our methods. 
 
  \subsection{Accuracy test}\label{accu}
Based on different choices of the parameters $a,b,c,d$ studied in Section \ref{thms}, several accuracy tests are carried out to compare the numerical solutions with the exact solutions (specified for each test below). Uniform meshes are considered. Nx denotes the total number of spatial cells, and Nt denotes the number of temporal steps. Periodic boundary conditions are applied for each examples. Note that some exact solutions, for instance the solitary-wave solutions, are not periodic in space. Thanks to their exponential decaying property, their values at two boundaries are very small, hence one can treat them as periodic functions and use this trick to check for accuracy. We choose the final time $T$ to be small such that the error due to periodic boundary conditions could be negligible.
We implemented the LDG schemes with both $k=1$ and $k=2$, which correspond to piecewise linear basis and piecewise quadratic  basis functions, respectively. 

\textbf{Case 1 (Theorem \ref{thm1}).} First, we consider the parameters 
$a=-\frac{7}{30},~ b=\frac{7}{15},~ c=-\frac{2}{5},~ d=\frac{1}{2}$, which matches the assumptions in Theorem \ref{thm1}.
For such choice of parameters, the following traveling-wave exact solutions have been provided in \cite{MC} 
   \begin{eqnarray}\label{eq_sec5.1}
     \begin{cases}
\eta(x,t)=\frac{3}{8}\sech^2\Big(\frac{1}{2}\sqrt{\frac{5}{7}}\big(x-20-\frac{5\sqrt{2}}{6}t\big)\Big), &\\
\\
u(x,t)=\frac{1}{2\sqrt{2}}\sech^2\Big(\frac{1}{2}\sqrt{\frac{5}{7}}\big(x-20-\frac{5\sqrt{2}}{6}t\big)\Big).&
\end{cases}
   \end{eqnarray}
The computational domain is set as $I=[0,40]$.
The initial conditions of $u(x,0)$ and $\eta(x,0)$ are obtained by setting $t = 0$ in \eqref{eq_sec5.1}, and the numerical error are computed at the final stopping time $T = 0.8$. The numerical error in $L^2$ and $L^\infty$ norms and the corresponding convergence rates are presented in Table \ref{e11} for $k=1$ and in Table \ref{e12} for $k=2$. Under both cases, the optimal convergence orders, i.e., $k+1$-th order, in both $u$ and $\eta$ can be observed, which indicated that the error estimate of $k+1/2$-th order in Theorem \ref{thm1} is not sharp. As we can observe that the lost of $1/2$ accuracy in the proof is due to the third order spatial derivative terms $u_{xxx}$, $\eta_{xxx}$, and one may try the idea presented in \cite{XS2012} to provide optimal error estimate for linear high order wave equations. The main idea there was to derive energy stability for the various auxiliary variables via using the scheme and its time derivatives with different test functions. Due to the existence of the nonlinear terms, we expect this procedure to be very cumbersome, and will investigate this in a future work. 
   
   \begin{table}[H]
     \caption{Numerical error and convergence rates of Case 1, with $k=1$.}
     \centering
     \begin{tabular}{|c|c|c|c|c|c|c|c|c|c|}
       \hline
       Nx&Nt&$\norm{e_u}_{L^2}$&rate&$\norm{e_\eta}_{L^2}$&rate&$\norm{e_u}_{\infty}$&rate&$\norm{e_\eta}_{\infty}$&rate\\
       \hline
              20&20&0.05225&0&0.05056&0&0.02611&0&0.01857&0\\
       \hline
       40&40&0.01359&1.9426&0.01211&2.0620&0.01433&0.8653&9.159E-3&1.0196\\
       \hline
       80&80&2.952E-3&2.2029&2.926E-3&2.0492&4.152E-3&1.7878&3.895E-3&1.2377\\
       \hline
       160&160&6.805E-4&2.1172&7.078E-4&2.0474&1.167E-3&1.8310&1.208E-3&1.6886\\
       \hline
       320&320&1.657E-4&2.0380&1.748E-4&2.0174&3.105E-4&1.9099&3.273E-4&1.8842\\
       \hline
     \end{tabular}
     \label{e11}
   \end{table}
   
   \begin{table}[H]
     \caption{Numerical error and convergence rates of Case 1, with $k=2$.}
 \centering
 \begin{tabular}{|c|c|c|c|c|c|c|c|c|c|}
\hline 
Nx&Nt&$\norm{e_u}_{L^2}$&rate&$\norm{e_\eta}_{L^2}$&rate&$\norm{e_u}_{\infty}$&rate&$\norm{e_\eta}_{\infty}$&rate\\
\hline 
20&20&5.578E-3&0&6.824E-3&0&3.188E-3&0&4.951E-3&0\\
\hline
40&40&7.203E-4&2.9532&7.834E-4&3.1228&1.202E-3&1.4071&8.763E-4&2.4983\\
\hline
80&80&9.575E-5&2.9113&9.584E-5&3.0312&1.940E-4&2.6319&1.727E-4&2.3428\\
\hline
160&160&1.276E-5&2.9078&1.321E-5&2.8589&2.695E-5&2.8472&2.690E-5&2.6830\\
\hline
320&320&1.645E-6&2.9557&1.732E-6&2.9311&3.498E-6&2.9456&3.645E-6&2.8836\\
\hline
\end{tabular}
\label{e12}
\end{table}

\textbf{Case 2 (Theorem \ref{thm2} with $a=c\neq 0$).} 
Here, we consider the parameters 
$a=c=\frac{1}{12},~b=\frac{1}{18},~d=\frac{1}{9}$, which matches the assumptions in Theorem \ref{thm2}.
As the exact solution is not available, we consider the following manufactured ``exact'' solution of the form 
\begin{eqnarray*}
  \begin{cases}
    \eta(x,t)=\cos(x+t), &\\
    u(x,t)=\sin(x+t),
  \end{cases}
\end{eqnarray*}
and modify the abcd Boussinesq system to be
\begin{equation*}
  \begin{cases}
\eta_t+u_x+(\eta u)_x+\frac{1}{12}u_{xxx}-\frac{1}{18}\eta_{xxt}=\cos(2x+2t)+\frac{11}{12}\cos(x+t)-\frac{19}{18}\sin(x+t), &\\
\\
u_t+\eta_x+\big(\frac{u^2}{2}\big)_x+\frac{1}{12}\eta_{xxx}-\frac{1}{9}u_{xxt}=\frac{1}{2}\sin(2x+2t)+\frac{10}{9}\cos(x+t)-\frac{11}{12}\sin(x+t),
\end{cases}
\end{equation*}
by adding some source terms on the right-hand side.
The computational domain is set as $I=[0,2\pi]$.
The initial conditions of $u(x,0)$ and $\eta(x,0)$ are obtained by setting $t = 0$ in the ``exact'' solutions, and the numerical error are computed at the final stopping time $T = {\pi}/{50}$. The numerical error in $L^2$ and $L^\infty$ norms and the corresponding convergence rates are presented in Table \ref{e21} for $k=1$ and in Table \ref{e22} for $k=2$. Again, the optimal convergence orders, i.e., $k+1$-th order, in both $u$ and $\eta$ can be observed, which indicated that the error estimate of $k+1/2$-th order in Theorem \ref{thm2} is not sharp. The same remedy discussed in Case 1 may be applied to improve the error estimate. 

\begin{table}[H]
     \caption{Numerical error and convergence rates of Case 2, with $k=1$.}
\centering
\begin{tabular}{|c|c|c|c|c|c|c|c|c|c|}
 \hline
 Nx&Nt&$\norm{e_u}_{L^2}$&rate&$\norm{e_\eta}_{L^2}$&rate&$\norm{e_u}_{\infty}$&rate&$\norm{e_\eta}_{\infty}$&rate\\
\hline 
20&20&9.183E-3&0&0.01057&0&0.01409&0&0.01581&0\\
\hline
40&40&2.732E-3&1.7492&2.772E-3&1.9306&4.132E-3&1.7697&4.171E-3&1.9231\\
\hline
80&80&6.789E-4&2.0085&6.732E-4&2.0422&1.028E-3&2.0069&1.020E-3&2.0312\\
\hline
160&160&1.697E-4&2.0004&1.691E-4&1.9935&2.571E-4&1.9995&2.563E-4&1.9931\\
\hline
320&320&4.240E-5&2.0005&4.232E-5&1.9980&6.427E-5&2.0000&6.418E-5&1.9979\\
\hline
 \end{tabular}
 \label{e21}
\end{table}

\begin{table}[H]
     \caption{Numerical error and convergence rates of Case 2, with $k=2$.}
\centering
 \begin{tabular}{|c|c|c|c|c|c|c|c|c|c|}
\hline 
Nx&Nt&$\norm{e_u}_{L^2}$&rate&$\norm{e_\eta}_{L^2}$&rate&$\norm{e_u}_{\infty}$&rate&$\norm{e_\eta}_{\infty}$&rate\\
\hline 
20&20&3.395E-4&0&3.281E-4&0&6.426E-4&0&4.727E-4&0\\
\hline
40&40&3.705E-5&3.1958&3.671E-5&3.1599&6.561E-5&3.2920&6.795E-5&2.7984\\
\hline
80&80&4.262E-6&3.1200&4.510E-6&3.0252&7.926E-6&3.0492&8.431E-6&3.0106\\
\hline
160&160&5.413E-7&2.9769&5.422E-7&3.0562&1.010E-6&2.9721&1.015E-6&3.0539\\
\hline
320&320&6.747E-8&3.0042&6.734E-8&3.0093&1.264E-7&2.9991&1.259E-7&3.0109\\
\hline
\end{tabular}
\label{e22}
\end{table}

\textbf{Case 3 (Theorem \ref{thm2} with $a=c=0$).}  
In Theorem \ref{thm2}, we are able to prove the optimal convergence rate when $a=c=0$, under which case the abcd Boussinesq system reduces to the coupled BBM equations. In the following test, we choose the parameters $a=c=0,~b=d=\frac{1}{6}$, 
and the corresponding traveling-wave exact solutions, provided in \cite{JX}, are given by
\begin{equation*}
  \begin{cases}
     \eta(x,t)=\frac{15}{4}\Big\{-2+\cosh\Big(\sqrt{\frac{18}{5}}\big(x-20-\frac{5}{2}t\big)\Big)\Big\}
     \sech^4\Big(\sqrt{\frac{9}{10}}\big(x-20-\frac{5}{2}t\big)\Big),&\\
     \\    
     u(x,t)=\frac{15}{2}\sech^2\Big(\sqrt{\frac{9}{10}}\big(x-\frac{L}{2}-\frac{5}{2}t\big)\Big).
  \end{cases}
\end{equation*}
The computational domain is set as $I=[0,40]$.
The initial conditions of $u(x,0)$ and $\eta(x,0)$ are obtained by setting $t = 0$ in the exact solutions, and the numerical error are computed at the final stopping time $T = 0.01$. The numerical error in $L^2$ and $L^\infty$ norms and the corresponding convergence rates are presented in Table \ref{e23} for $k=1$ and in Table \ref{e24} for $k=2$. The optimal convergence orders in both $u$ and $\eta$ can be observed, which matches the optimal error estimate analysis in Theorem \ref{thm2}. 

\begin{table}[htb]
     \caption{Numerical error and convergence rates of Case 3, with $k=1$.}
\centering
\begin{tabular}{|c|c|c|c|c|c|c|c|c|c|}
\hline 
Nx&Nt&$\norm{e_u}_{L^2}$&rate&$\norm{e_\eta}_{L^2}$&rate&$\norm{e_u}_{\infty}$&rate&$\norm{e_\eta}_{\infty}$&rate\\
\hline 
20&20&0.6819&0&1.8824&0&0.7956&0&1.5628&0\\
\hline
40&40&0.3377&1.0139&0.2893&2.7018&0.4860&0.7111&0.3183&2.2956\\
\hline
80&80&0.1108&1.6077&0.1898&0.6084&0.2138&1.1846&0.3091&0.0424\\
\hline
160&160&0.02812&1.9784&0.05000&1.9241&0.06636&1.6879&0.1255&1.2999\\
\hline
320&320&7.057E-3&1.9944&0.01261&1.9872&0.01745&1.9268&0.03474&1.8534\\
\hline
\end{tabular}
\label{e23}
\end{table}
\begin{table}[H]
     \caption{Numerical error and convergence rates of Case 3, with $k=2$.}
\centering
\begin{tabular}{|c|c|c|c|c|c|c|c|c|c|}
\hline 
Nx&Nt&$\norm{e_u}_{L^2}$&rate&$\norm{e_\eta}_{L^2}$&rate&$\norm{e_u}_{\infty}$&rate&$\norm{e_\eta}_{\infty}$&rate\\
\hline 
20&20&0.4594&0&0.4591&0&0.7066&0&0.5598&0\\
\hline
40&40&0.1025&2.1637&0.2525&0.8626&0.1830&1.9491&0.4554&0.2979\\
\hline
80&80&0.01054&3.2832&0.03005&3.0707&0.02108&3.1175&0.07503&2.6015\\
\hline
160&160&1.335E-3&2.9808&3.668E-3&3.0341&3.300E-3&2.6756&0.01056&2.8282\\
\hline
320&320&1.676E-4&2.9933&4.629E-4&2.9862&4.195E-4&2.9758&1.350E-3&2.9682\\
\hline
\end{tabular}
\label{e24}
\end{table}

\textbf{Case 4 (Theorem \ref{thm3}).}   
Here, we consider the parameters 
$a=c=\frac{1}{9},~b=\frac{1}{9},~d=0$, which matches the assumptions in Theorem \ref{thm3}.
As the exact solution is not available, we consider the following manufactured ``exact'' solution of the form 
\begin{eqnarray*}
  \begin{cases}
    \eta(x,t)=\cos(x+t), &\\
    u(x,t)=\sin(x+t),
  \end{cases}
\end{eqnarray*}
and modify the abcd Boussinesq system to be
\begin{equation*}
  \begin{cases}
\eta_t+u_x+(\eta u)_x+\frac{1}{9}u_{xxx}-\frac{1}{9}\eta_{xxt}=\cos(2x+2t)+\frac{8}{9}\cos(x+t)-\frac{10}{9}\sin(x+t), &\\
u_t+\eta_x+\big(\frac{u^2}{2}\big)_x+\frac{1}{9}\eta_{xxx}=\frac{1}{2}\cos(2x+2t)+\cos(x+t)-\frac{8}{9}\sin(x+t).&
\end{cases}
\end{equation*}
by adding some source terms on the right-hand side.
The computational domain is set as $I=[0,8\pi]$.
The initial conditions of $u(x,0)$ and $\eta(x,0)$ are obtained by setting $t = 0$ in the ``exact'' solutions, and the numerical error are computed at the final stopping time $T = 0.04$.
 The numerical error in $L^2$ and $L^\infty$ norms and the corresponding convergence rates are presented in Table \ref{e31} for $k=1$ and in Table \ref{e32} for $k=2$. The optimal convergence orders, i.e., $k+1$-th order, in both $u$ and $\eta$ can be observed. 

\begin{table}[H]
     \caption{Numerical error and convergence rates of Case 4, with $k=1$.}
  \centering
\begin{tabular}{|c|c|c|c|c|c|c|c|c|c|}
\hline 
Nx&Nt&$\norm{e_u}_{L^2}$&rate&$\norm{e_\eta}_{L^2}$&rate&$\norm{e_u}_{\infty}$&rate&$\norm{e_\eta}_{\infty}$&rate\\
\hline 
20&1000&0.2329&0&0.2327&0&0.1329&0&0.1375&0\\
\hline
40&2000&0.05657&2.0418&0.05645&2.0436&0.04002&1.7315&0.04048&1.7643\\
\hline
80&4000&0.01779&1.6694&0.01714&1.7194&0.01318&1.6018&0.01302&1.6361\\
\hline
160&8000&5.131E-3&1.7935&5.044E-3&1.7649&3.841E-3&1.7794&3.818E-3&1.7701\\
\hline
320&16000&1.339E-3&1.9382&1.331E-3&1.9221&1.007E-3&1.9319&1.005E-3&1.9259\\
\hline
\end{tabular} \label{e31}
\end{table}

\begin{table}[H]
     \caption{Numerical error and convergence rates of Case 4, with $k=2$.}
\centering
\begin{tabular}{|c|c|c|c|c|c|c|c|c|c|}
\hline 
Nx&Nt&$\norm{e_u}_{L^2}$&rate&$\norm{e_\eta}_{L^2}$&rate&$\norm{e_u}_{\infty}$&rate&$\norm{e_\eta}_{\infty}$&rate\\
\hline 
20&1000&0.03777&0&0.03034&0&9.041E-3&0&9.676E-3&0\\
\hline
40&2000&5.215E-3&2.8563&3.617E-3&3.0683&7.160E-4&3.6584&1.651E-3&2.5512\\
\hline
80&4000&5.268E-4&3.3074&4.844E-4&2.9005&2.896E-4&1.3062&3.663E-4&2.1721\\
\hline
160&8000&6.488E-5&3.0214&6.430E-5&2.9132&5.508E-5&2.3943&5.629E-5&2.7023\\
\hline
320&16000&8.480E-6&2.9356&8.372E-6&2.9411&7.814E-6&2.8173&7.663E-6&2.8767\\
\hline
\end{tabular}\label{e32}
\end{table}

\textbf{Case 5 (Theorem \ref{thm3.5}).}   
Here, we consider the parameters 
$a=c=b=0,~d=\frac{1}{6}$, which matches the assumptions in Theorem \ref{thm3.5}.
An exact solution is presented in \cite{CB} as follows
\begin{eqnarray*}
  \begin{cases}
    \eta(x,t)=-1 &\\
    u(x,t)=\frac23+\sech^2\Big(\frac{1}{\sqrt{2}}(x-20-t)\Big),
  \end{cases}
\end{eqnarray*}
The computational domain is set as $I=[0,40]$.
The initial conditions of $u(x,0)$ and $\eta(x,0)$ are obtained by setting $t = 0$ in the exact  solutions,
 and the numerical error are computed at the final stopping time $T = 0.01$.
 The numerical error in $L^2$ and $L^\infty$ norms and the corresponding convergence rates are presented in Table \ref{e3.51} for $k=1$ and in Table \ref{e3.52} 
 for $k=2$.  
For this case, we can observe that $\norm{e_u}_{L^2}$ demonstrates a $k+1$th order of accuracy, and 
$\norm{e_\eta}_{L^2}$ presents the machine error at the level of $10^{-14}$. 
This is due to the fact that the exact solution $\eta(t,x)=-1$ is a constant,
 hence its approximation via polynomial is exact and the update equation for $\eta$ is also exact.
\begin{table}[H]
     \caption{Numerical error and convergence rates of Case 5, with $k=1$.}
  \centering
\begin{tabular}{|c|c|c|c|c|c|c|c|c|c|}
\hline 
Nx&Nt&$\norm{e_u}_{L^2}$&rate&$\norm{e_\eta}_{L^2}$&$\norm{e_u}_{\infty}$&rate&$\norm{e_\eta}_{\infty}$\\
\hline 
20&20&0.05729&0&7.615E-16&0.07459&0&2.220E-16\\
\hline
40&40&0.03581&0.6779&7.322E-16&0.05536&0.4301&2.220E-16\\
\hline
80&80&9.578E-3&1.9027&7.167E-16&0.01849&1.5820&2.220E-16\\
\hline
160&160&2.413E-3&1.9889&7.347E-16&5.020E-3&1.8812&2.220E-16\\
\hline
320&320&6.045E-4&1.9971&7.697E-16&1.290E-3&1.9603&2.220E-16\\
\hline
\end{tabular} \label{e3.51}
\end{table}

\begin{table}[H]
     \caption{Numerical error and convergence rates of Case 5, with $k=2$.}
\centering
\begin{tabular}{|c|c|c|c|c|c|c|c|c|c|}
\hline 
Nx&Nt&$\norm{e_u}_{L^2}$&rate&$\norm{e_\eta}_{L^2}$&$\norm{e_u}_{\infty}$&rate&$\norm{e_\eta}_{\infty}$\\
\hline 
20&20&0.04538&0&1.772E-15&0.06566&0&1.221E-15\\
\hline
40&40&5.871E-3&2.9504&2.296E-15&0.01028&2.6750&1.221E-15\\
\hline
80&80&6.758E-4&3.1190&4.429E-15&1.342E-3&2.9372&3.109E-15\\
\hline
160&160&8.530E-5&2.9860&4.197E-14&1.819E-4&2.8837&3.109E-14\\
\hline
320&320&1.069E-5&2.9961&6.027E-14&2.313E-5&2.9748&3.020E-14\\
\hline
\end{tabular}\label{e3.52}
\end{table}

\textbf{Case 6 (Theorem \ref{thm4}).}  
The parameters $a=0,~b=d=\frac{1}{3},~c=-\frac{1}{3}$, which matches the assumptions in Theorem \ref{thm4}, are chosen in
this test. A set of exact solutions, studied in \cite{CB}, takes the form of 
\begin{eqnarray*}
  \begin{cases}
    \eta(x,t)=-1, &\\
    u(x,t)=1+6\sech^2\Big(\frac{1}{\sqrt 2}\big(x-20-3t\big)\Big).
  \end{cases}
\end{eqnarray*}
The computational domain is set as $I=[0,40]$.
The initial conditions of $u(x,0)$ and $\eta(x,0)$ are obtained by setting $t = 0$ in the exact solutions, 
and the numerical error are computed at the final stopping time $T = 0.01$. The numerical error in $L^2$ and $L^\infty$ 
norms and the corresponding convergence rates are presented in Table \ref{e41} for $k=1$ and in Table \ref{e42} for $k=2$. 
 We can observe that $\norm{e_u}_{L^2}$ demonstrates a $k+1$th order of accuracy, 
 and, for the same reason as in Case 5, $\norm{e_\eta}_{L^2}$ 
 presents the machine error. 

\begin{table}[H]
     \caption{Numerical error and convergence rates of Case 6, with $k=1$.}
\centering
\begin{tabular}{|c|c|c|c|c|c|c|c|}
\hline 
Nx&Nt&$\norm{e_u}_{L^2}$&rate&$\norm{e_\eta}_{L^2}$&$\norm{e_u}_{\infty}$&rate&$\norm{e_\eta}_{\infty}$\\
\hline
20&20&0.3461&0&7.608E-16&0.3831&0&2.220E-16\\
\hline
40&40&0.2149&0.6873&7.322E-16&0.3037&0.3354&2.220E-16\\
\hline
80&80&0.05755&1.9011&7.175E-16&0.1071&1.5031&2.220E-16\\
\hline
160&160&0.01450&1.9884&7.460E-10&0.03042&1.8164&2.220E-16\\
\hline
320&320&3.633E-3&1.9969&8.008E-16&7.792E-3&1.9647&2.220E-16\\
\hline
\end{tabular}\label{e41}
\end{table}

\begin{table}[H]
     \caption{Numerical error and convergence rates of Case 6, with $k=2$.}
\centering
\begin{tabular}{|c|c|c|c|c|c|c|c|}
\hline 
Nx&Nt&$\norm{e_u}_{L^2}$&rate&$\norm{e_\eta}_{L^2}$&$\norm{e_u}_{\infty}$&rate&$\norm{e_\eta}_{\infty}$\\
\hline 
20&20&0.2722&0&1.800E-15&0.3726&0&1.332E-15\\
\hline
40&40&0.03517&2.9519&2.313E-15&0.06220&2.5826&1.332E-15\\
\hline
80&80&4.064E-3&3.1137&4.538E-15&8.324E-3&2.9016&3.109E-15\\
\hline
160&160&5.132E-4&2.9852&4.278E-14&1.113E-3&2.9022&3.331E-14\\
\hline
320&320&6.432E-5&2.9961&6.325E-14&1.406E-4&2.9850&3.331E-14\\
\hline
\end{tabular}\label{e42}
\end{table}

\textbf{Case 7 (KdV-KdV system).} In this last example for accuracy test, we consider the case of $b=d=0$, 
under which case the abcd Boussinesq system reduces to the coupled KdV-KdV equations. Note that we were not able to 
provide an error estimate analytically for this system. Let $b=d=0,~a=c=\frac{1}{6}$, there exits exact solutions of the form \cite{MC}
\begin{equation*}
  \begin{cases} 
    \eta(x,t)=-1+\frac{3}{2}\sech^2\Big(\sqrt{\frac{3}{2}}(x-20-\sqrt{2}t)\Big),&\\
    \\
    u(x,t)=\frac{3}{\sqrt{2}}\sech^2\Big(\sqrt{\frac{3}{2}}(x-20-\sqrt{2}t)\Big).
  \end{cases}
\end{equation*}
The computational domain is set as $I=[0,40]$.
The initial conditions of $u(x,0)$ and $\eta(x,0)$ are obtained by setting $t = 0$ in the exact solutions, and the numerical error are computed at the final stopping time $T = 0.01$. The numerical error in $L^2$ and $L^\infty$ norms and the corresponding convergence rates are presented in Table \ref{e61} for $k=1$ and in Table \ref{e62} for $k=2$. The optimal convergence orders in both $u$ and $\eta$ can be observed numerically. 
This is consistent with the Galerkin approximation of the KdV-KdV equations in \cite{BDM}.

\begin{table}[H]
     \caption{Numerical error and convergence rates of Case 7, with $k=1$.}
\centering
\begin{tabular}{|c|c|c|c|c|c|c|c|c|c|}
\hline 
Nx&Nt&$\norm{e_u}_{L^2}$&rate&$\norm{e_\eta}_{L^2}$&rate&$\norm{e_u}_{\infty}$&rate&$\norm{e_\eta}_{\infty}$&rate\\
\hline 
20&200&0.3526&0&0.2514&0&0.3645&0&0.2503&0\\
\hline
40&400&0.09791&1.8486&0.08920&1.4950&0.1663&1.1325&0.1244&1.0088\\
\hline
80&800&0.06758&0.5349&0.05326&0.7440&0.1493&0.1558&0.1133&0.1351\\
\hline
160&1600&0.01877&1.8478&0.01352&1.9783&0.05732&1.3807&0.04096&1.4675\\
\hline
320&3200&4.820E-3&1.9618&3.423E-3&1.9815&0.01598&1.8340&0.01131&1.8563\\
\hline
\end{tabular}\label{e61}
\end{table}

\begin{table}[H]
     \caption{Numerical error and convergence rates of Case 7, with $k=2$.}
\centering
\begin{tabular}{|c|c|c|c|c|c|c|c|c|c|}
\hline 
Nx&Nt&$\norm{e_u}_{L^2}$&rate&$\norm{e_\eta}_{L^2}$&rate&$\norm{e_u}_{\infty}$&rate&$\norm{e_\eta}_{\infty}$&rate\\
\hline 
20&200&0.1449&0&0.1346&0&0.2870&0&0.2335&0\\
\hline
40&400&0.07251&0.9989&0.05373&1.3243&0.1599&0.8439&0.1178&0.9874\\
\hline
80&800&8.667E-3&3.0645&6.394E-3&3.0710&0.02188&2.8693&0.01675&2.8138\\
\hline
160&1600&1.087E-3&2.9953&7.816E-4&3.0322&3.703E-3&2.5629&2.701E-3&2.6326\\
\hline
320&3200&1.402E-4&2.9548&9.963E-5&2.9719&5.021E-4&2.8828&3.572E-4&2.9185\\
\hline
\end{tabular}\label{e62}
\end{table}

\subsection{Wave collisions}\label{colli}
Motivated by the experiments in \cite{BonaChen2016,CB}, we consider two examples of traveling-wave collisions in this section. The head-on collision, namely, two waves travel in opposite direction, is considered here. 

\subsubsection{Finite time blow up}\label{blowup}
First, we consider the experiment performed in \cite{BonaChen2016,CB} for the coupled BBM equations
with the choice of parameters $a=c=0,~b=d=\frac{1}{6}$.
The system admits the traveling wave exact solution of the form
\begin{eqnarray*}
    \begin{cases}
      \eta_{\pm}(x,t)=\frac{15}{2}\sech^2\Bigg(\frac{3}{\sqrt{10}}\Big(x\pm\frac52 t\Big)\Bigg)   
      -\frac{45}{4}\sech^4\Bigg(\frac{3}{\sqrt{10}}\Big(x\pm\frac52 t\Big)\Bigg),\\
\\
u_{\pm}(x,t)=\mp\frac{15}{2}\sech^2\Bigg(\frac{3}{\sqrt{10}}\Big(x\pm\frac52 t\Big)\Bigg),
    \end{cases}
\end{eqnarray*}
traveling at opposite direction. 
The initial conditions are set as
\begin{eqnarray*}
    &\eta(x,0)=\eta_+(x-x_+,0)+\eta_-(x-x_-,0),\\
    &u(x,0)=u_+(x-x_+,0)+u_-(x-x_-,0),
  \end{eqnarray*}
where $x_{\pm}=\pm7$ in the computational domain $I=[-14,14]$. Initially the waves are centered at $7$ and $-7$,
and will propagate towards each other. 

We use the LDG scheme with $k=2$ and $h = 0.175$  to simulate this example. 
The time step size is chosen to be $\Delta t=0.00107h$ up to $t=3.24$, and $\Delta t=0.00052h$ for $3.24\leq t\leq4.4$.
The numerical solutions $u_h$ and $\eta_h$ at various times (before the blow-up) are shown in Fig. \ref{fig1}.
We can observe that the $L_\infty$-norm of the numerical solutions $\eta_h$ at $t=4.4$ reaches 120, which indicates the possible
blow-up of the $\eta_h$, as well as the spatial derivative of $u_h$. This is consistent with the observations in \cite{BonaChen2016,CB}.

\begin{figure}[H]
\centering
\subfigure[$u_h$ at $t=0$]{\includegraphics[width=0.35\textwidth]{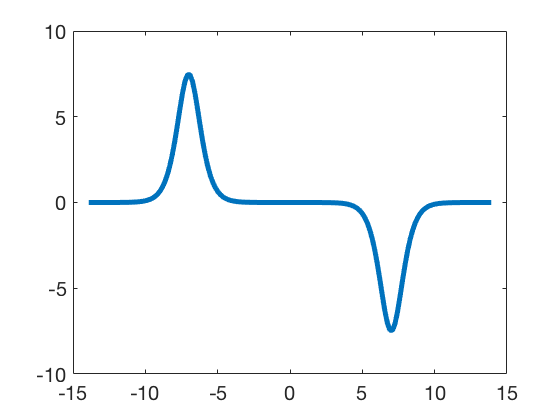}}
\subfigure[$\eta_h$ at $t=0$]{\includegraphics[width=0.35\textwidth]{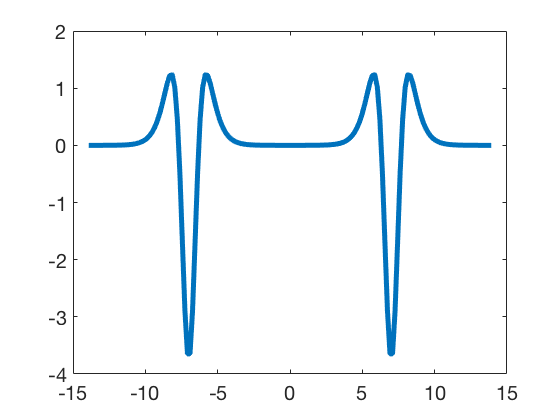}}
\subfigure[$u_h$ at $t=1.5$]{\includegraphics[width=0.35\textwidth]{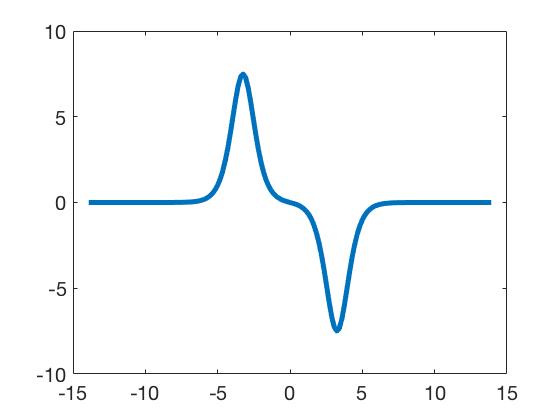}}
\subfigure[$\eta_h$ at $t=1.5$]{\includegraphics[width=0.35\textwidth]{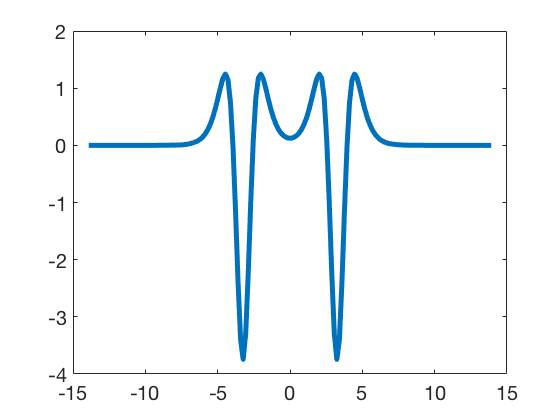}}
\subfigure[$u_h$ at $t=3.24$]{\includegraphics[width=0.35\textwidth]{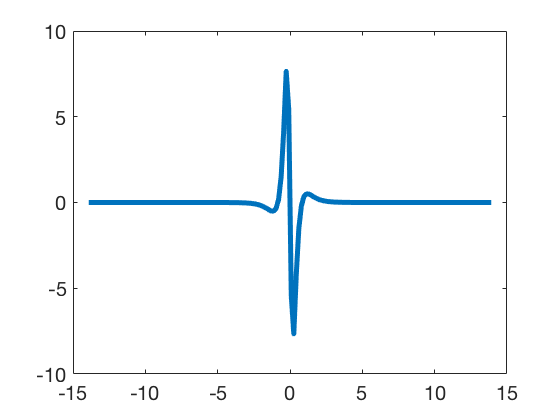}}
\subfigure[$\eta_h$ at $t=3.24$]{\includegraphics[width=0.35\textwidth]{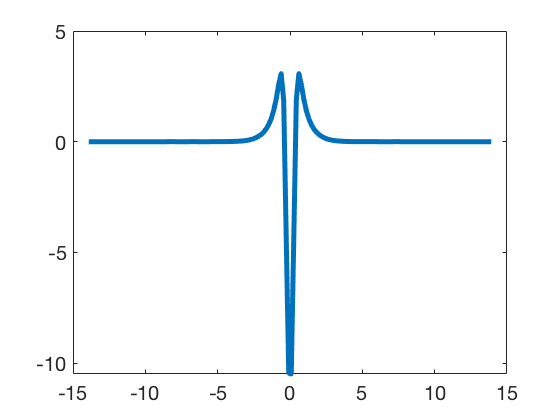}}
\subfigure[$u_h$ at $t=4.4$]{\includegraphics[width=0.35\textwidth]{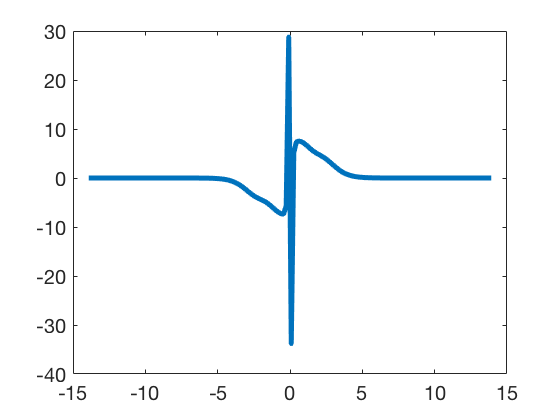}}
\subfigure[$\eta_h$ at $t=4.4$]{\includegraphics[width=0.35\textwidth]{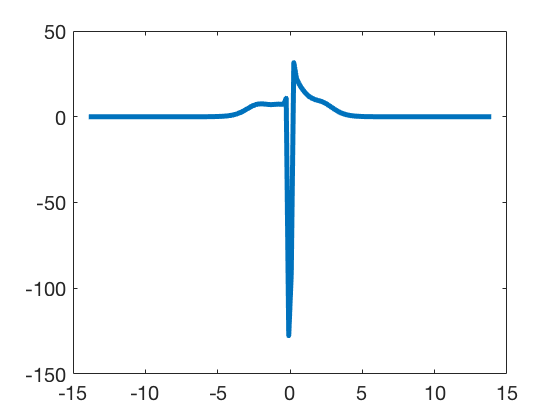}}
\caption{The numerical solution $u_h$ and $\eta_h$ of the finite time blow-up test at different times.}
\label{fig1}
\end{figure}

 \subsubsection{Head-on collision}\label{headon}

 We consider the system equations with the choice of parameters $a=-\frac{7}{30},b=\frac{7}{15}, c=-\frac{2}{5}, d=\frac{1}{2}$,
and the initial conditions  given in \cite{CB}:
 \begin{equation*}
     \eta(x,0)=\eta_+(x)+\eta_-(x), \qquad
     u(x,0)=u_+(x)+u_-(x),
 \end{equation*}
where 
 \begin{equation*}
   \begin{cases}
		\eta_{\pm}(x)=\frac{1}{\sqrt{8}}\sech^2\Big({\sqrt{\frac{5}{28}}}(x-x_{\pm})\Big),&\\
     \\
 	u_{\pm}(x)=\pm\frac{3}{8}\sech^2\Big({\sqrt{\frac{5}{28}}}(x+x_{\pm})\Big).
   \end{cases}
 \end{equation*}
Here the computational domain is set as $I=[-14,14]$, and $x_{\pm}=\pm 7$. Two waves travel towards each other, and collide
around $t= 6$. 

We use the LDG scheme with $k=2$ and $h = 0.175 $ to simulate this example. The time step size is chosen to be $\Delta t\approx0.00214h$.
The numerical solutions $u_h$ and $\eta_h$ at various times (before and after the collision) are shown in Fig. \ref{fig3}.
Numerically, we can observe that these two waves merge, and then split up and continue to travel independently.
 \begin{figure}[H]
\centering
\subfigure[$t=0$ for $u_h$]{\includegraphics[width=0.35\textwidth]{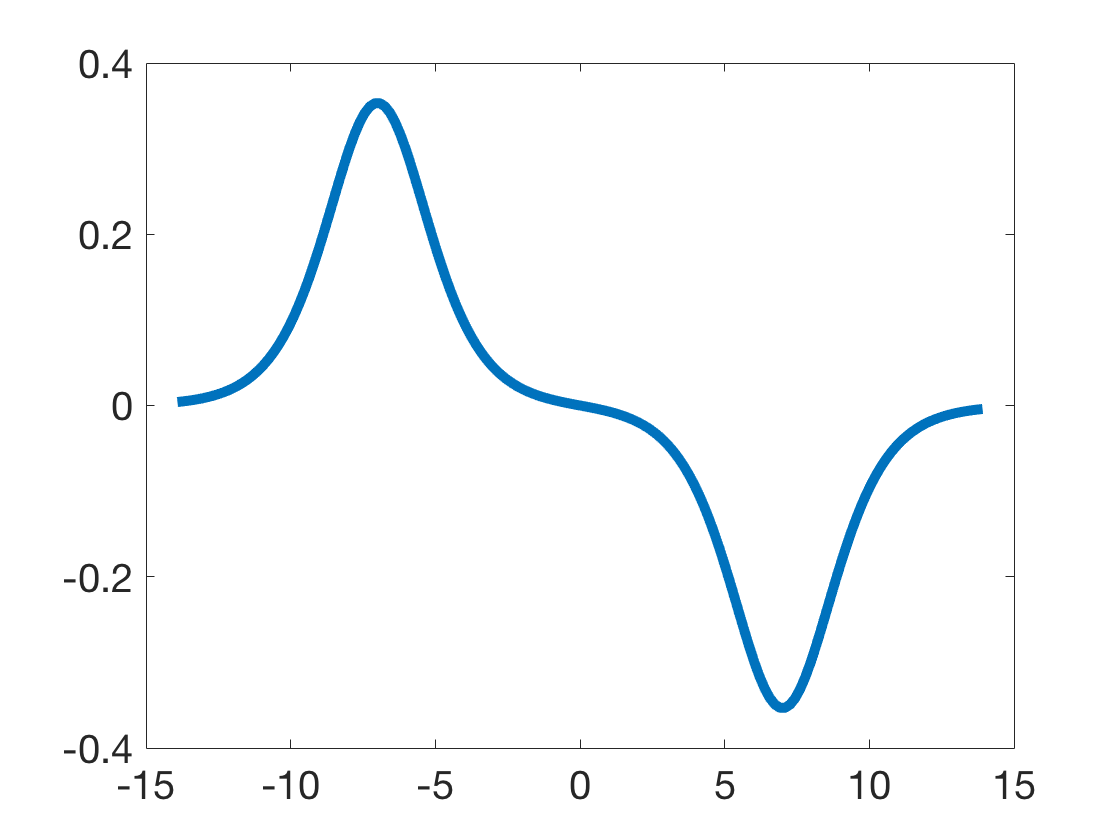}}
\subfigure[$t=0$ for $\eta_h$]{\includegraphics[width=0.35\textwidth]{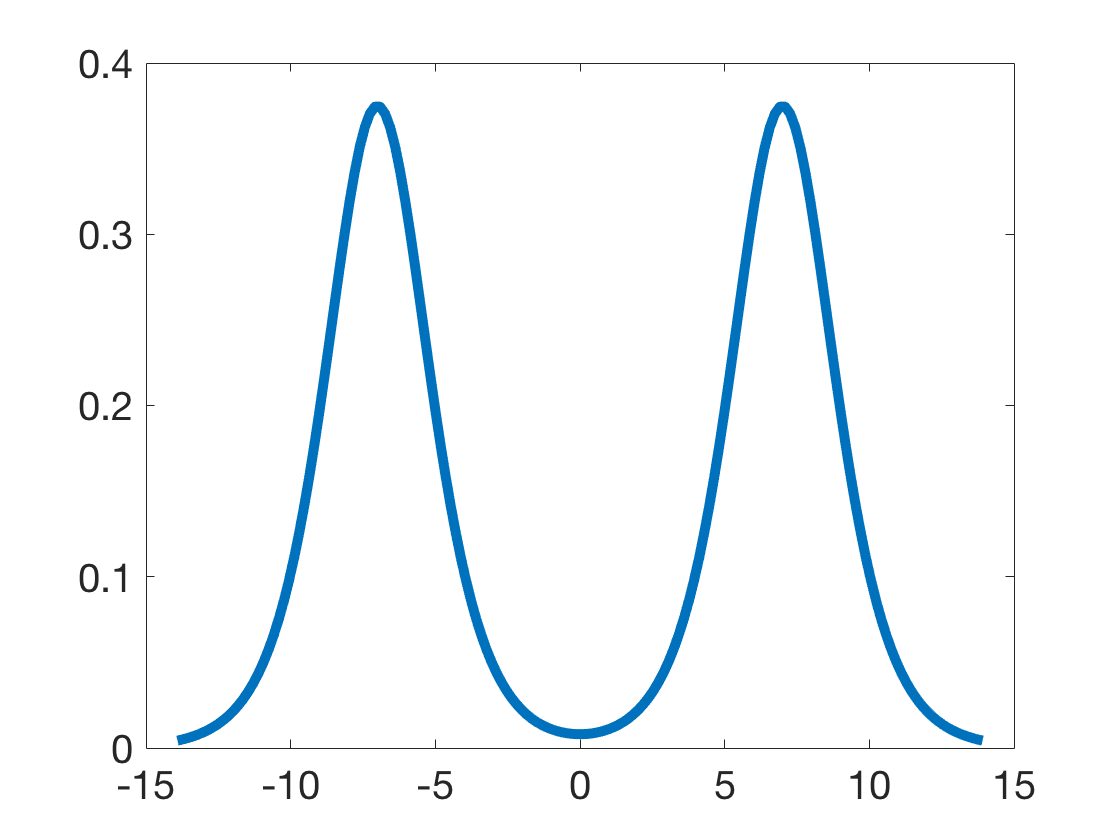}}
\subfigure[$t=3$ for $u_h$]{\includegraphics[width=0.35\textwidth]{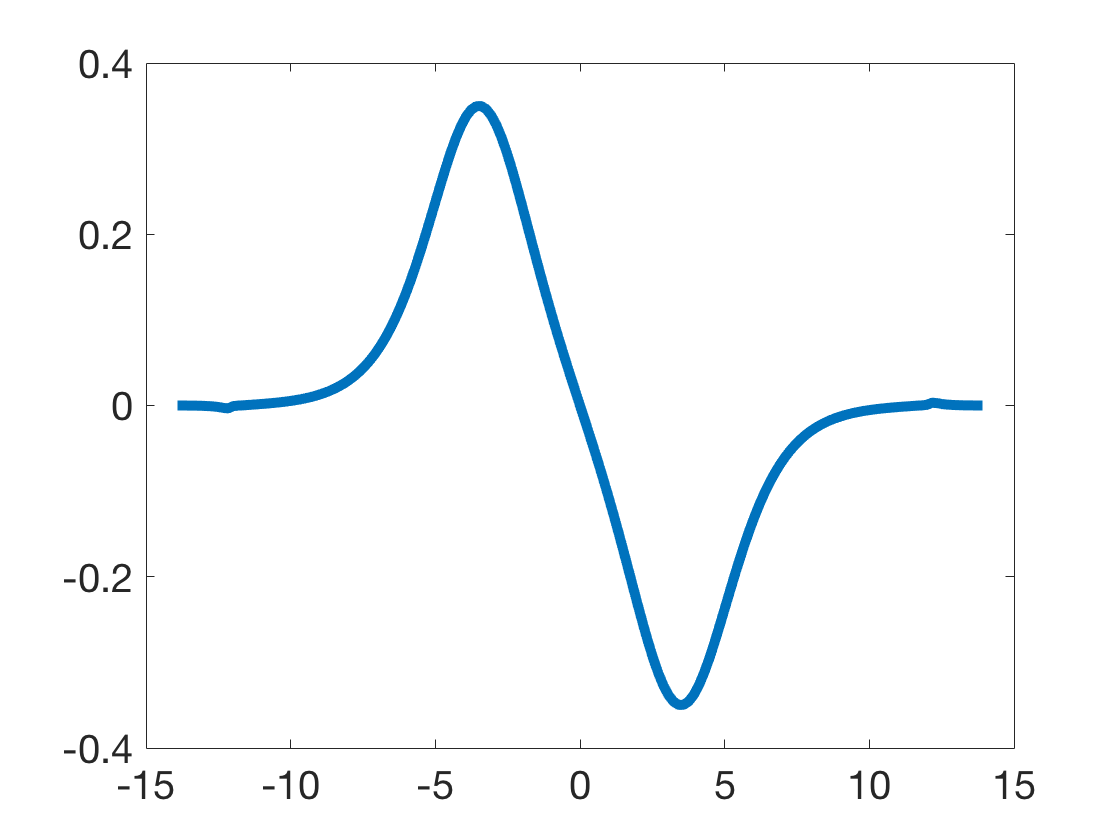}}
\subfigure[$t=3$ for $\eta_h$]{\includegraphics[width=0.35\textwidth]{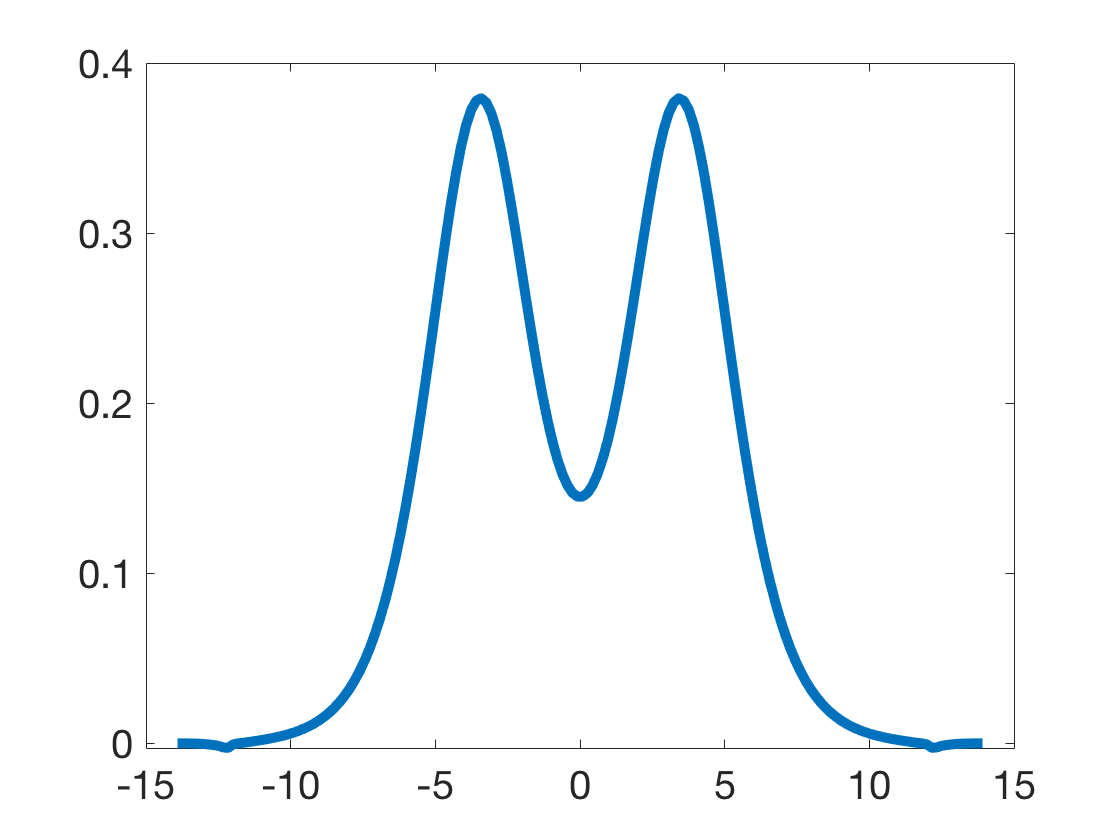}}
\end{figure}
\begin{figure}[H]
\centering
\subfigure[$t=6$ for $u_h$]{\includegraphics[width=0.35\textwidth]{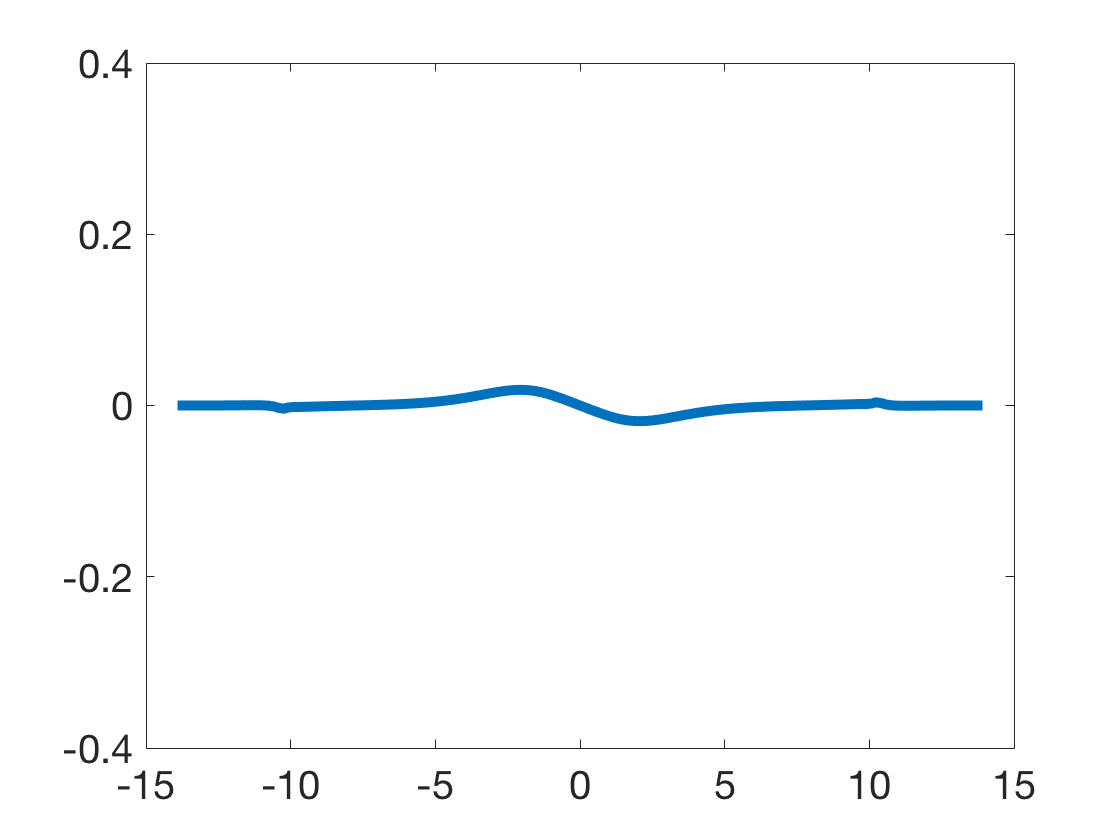}}
\subfigure[$t=6$ for $\eta_h$]{\includegraphics[width=0.35\textwidth]{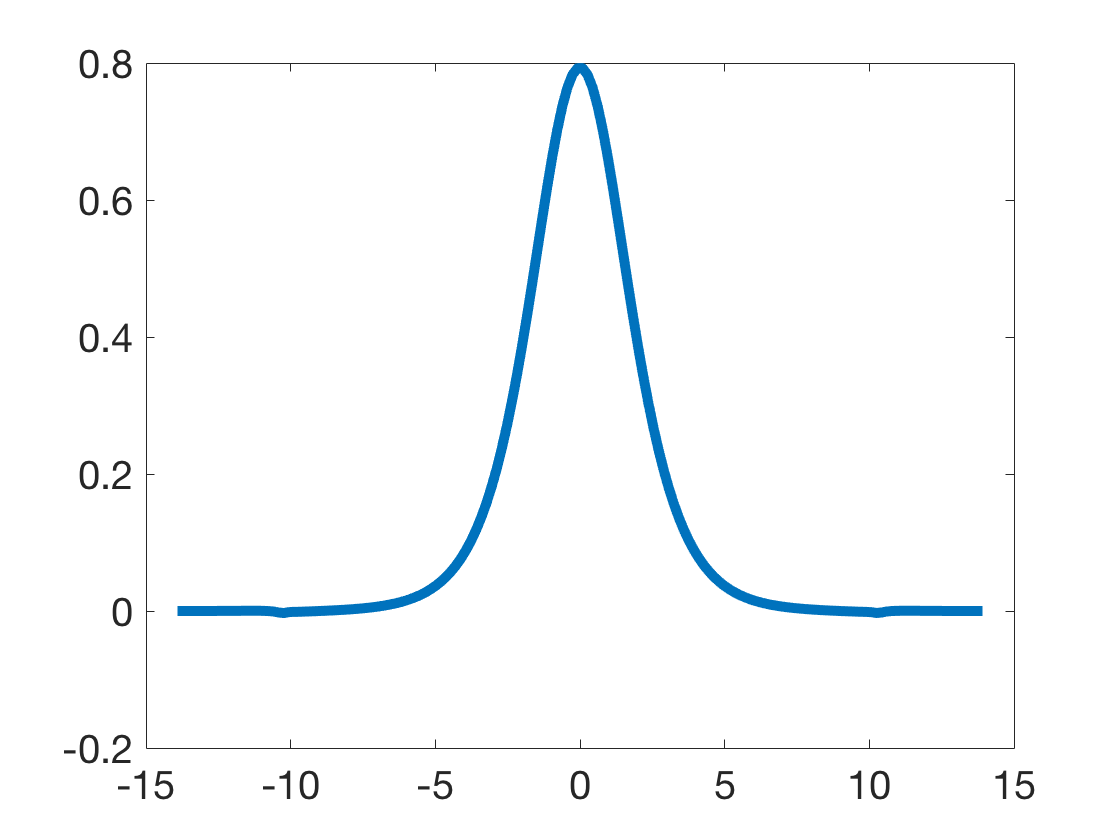}}
\subfigure[$t=9$ for $u_h$]{\includegraphics[width=0.35\textwidth]{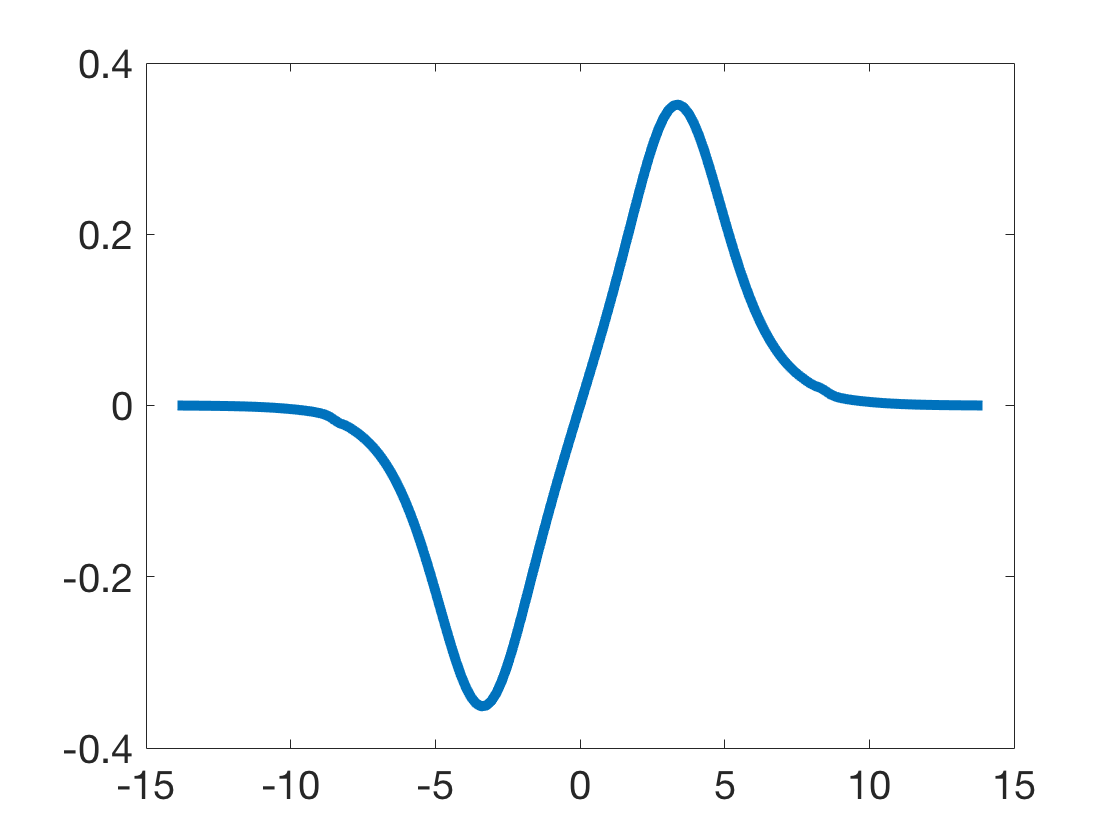}}
\subfigure[$t=9$ for $\eta_h$]{\includegraphics[width=0.35\textwidth]{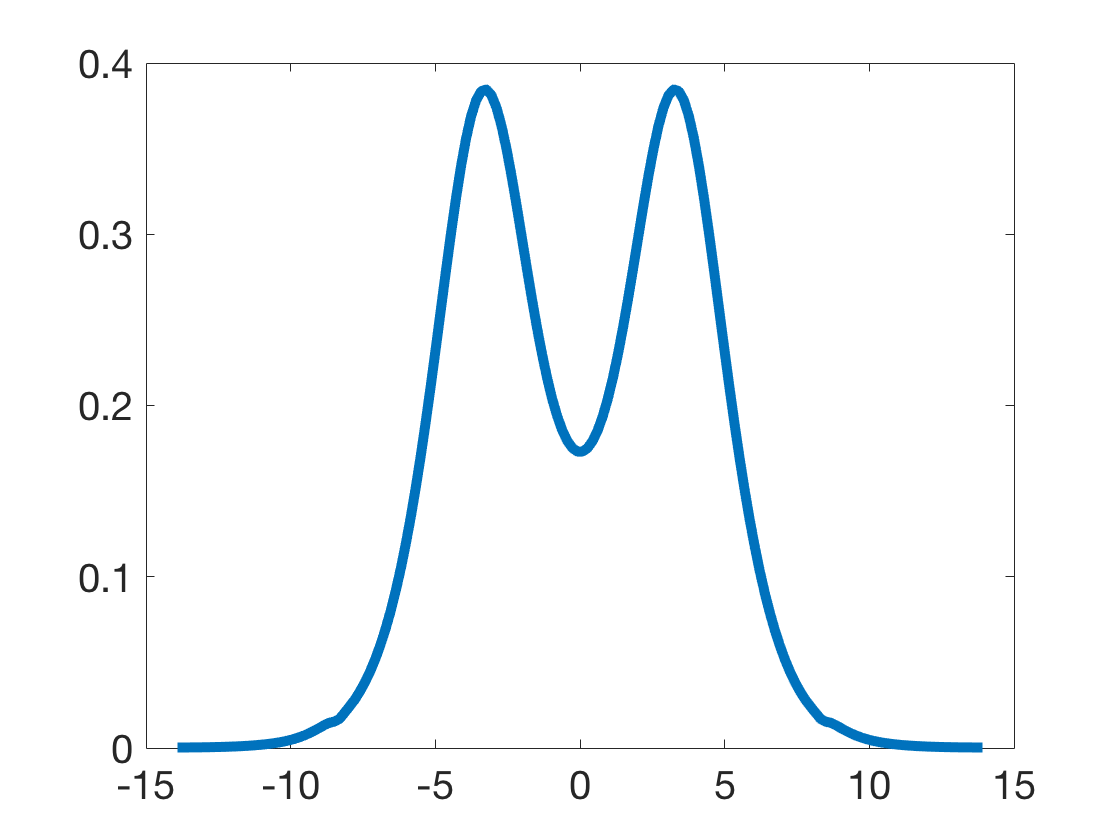}}
\subfigure[$t=12$ for $u_h$]{\includegraphics[width=0.35\textwidth]{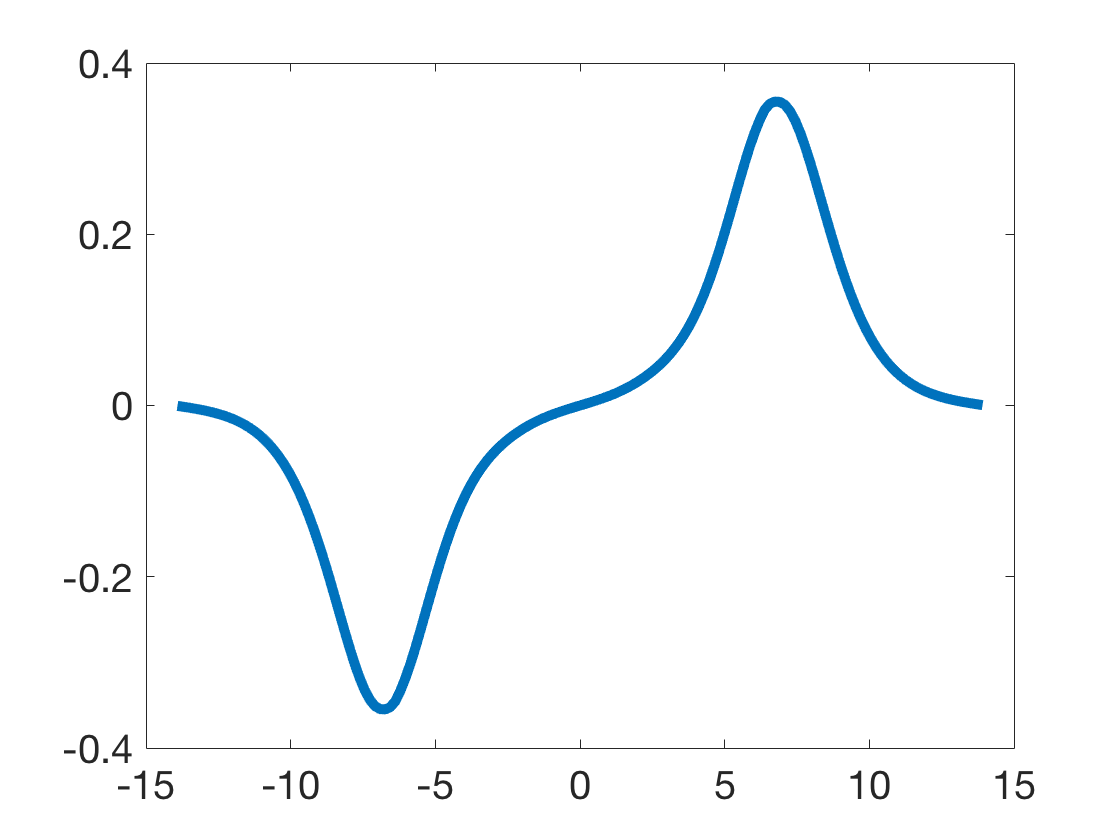}}
\subfigure[$t=12$ for $\eta_h$]{\includegraphics[width=0.35\textwidth]{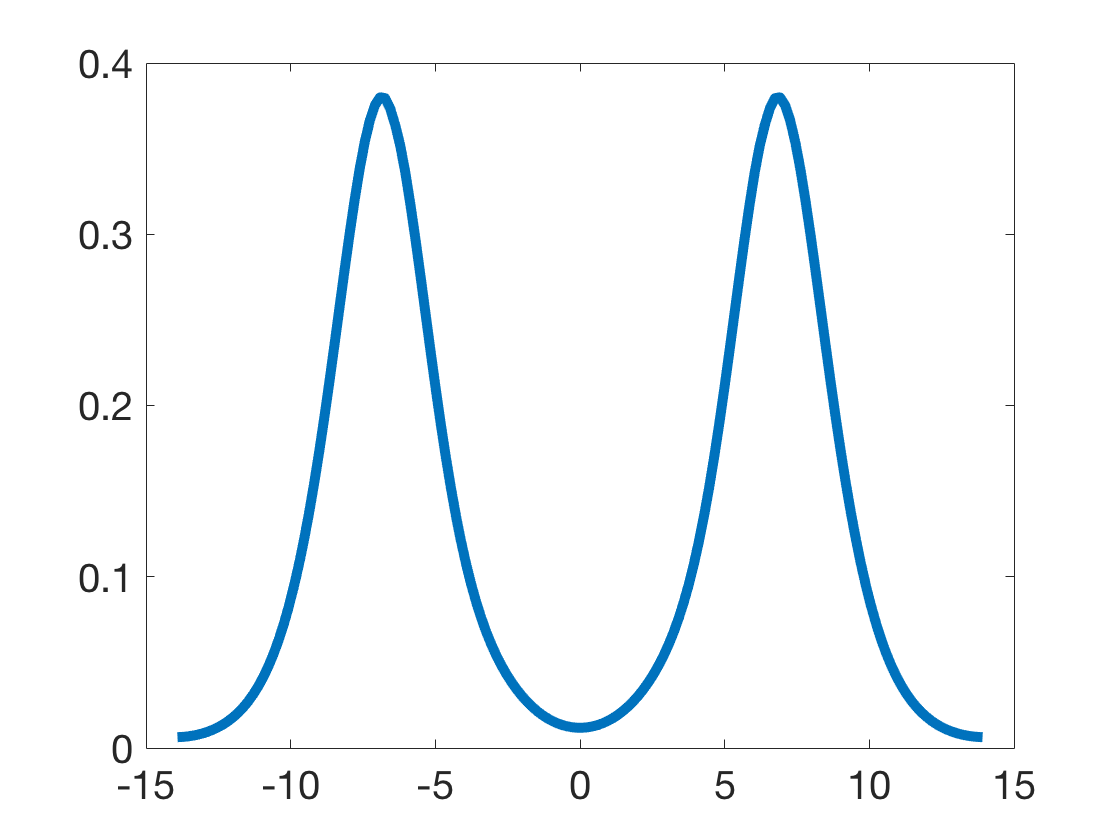}}
\caption{The numerical solution $u_h$ and $\eta_h$ of the head-on collision test at different times. }
\label{fig3}
 \end{figure}

 \section{Conclusion remark}\label{conclusion} \setcounter{equation}{0} \setcounter{table}{0} \setcounter{figure}{0}
In this paper we developed and analyzed the LDG methods for abcd Boussinesq system. 
By utilizing the connection between the error of the auxiliary and primary variables, we proved 
the optimal error estimate for our LDG scheme applied to the BBM-BBM system, and 
sub-optimal error estimate for a wide range of parameters $a,b,c,d>0$. Numerical experiments are provided 
to test the accuracy of our methods, and optimal error estimates are observed for all the cases, including KdV-KdV system 
which is not covered in our theoretical analysis. We also simulate wave collisions of traveling wave solutions 
and observe finite time blow-up behavior of the numerical solutions. 

 \section*{Acknowledgments}
The work of J. Sun and Y. Xing is partially sponsored by NSF grant DMS-1753581.
 
\section*{Conflict of Interest}

The authors declare that they have no conflict of interest.
 


\begin{thebibliography}{99}\addcontentsline{toc}{chapter}{Bibliography}

    \bibitem{Amick} C. J. Amick, \emph{Regularity and uniqueness of solutions to the Boussinesq system of equations,}
    Journal of Differential Equations, 54:231-247, 1984
    
    \bibitem{Antonopoulos} D. C. Antonopoulos, V. A. Dougalis, D. E. 
    Mitsotakis, \emph{Galerkin approximations of periodic solutions of Boussinesq 
    systems}, Bulletin of the Greek Mathematical Society, 57:13-30, 2010
    
    \bibitem{BonaChen1997} J. L. Bona, M. Chen, \emph{A Boussinesq system for two-way propagation of nonlinear dispersive 
    waves}, Physica D, 116:191-224, 1998
    
    \bibitem{BonaChen2016} J. L. Bona, M. Chen, \emph{Singular solutions of a Boussinesq system for water waves}, 
    Journal of  Mathematical Study, 49:205 - 220, 2016
    
    \bibitem{BCKX2013} J.L. Bona, H. Chen, O. Karakashian, Y. Xing, \emph{Conservative, discontinuous-Galerkin methods for the generalized Korteweg-de Vries equation}, Mathematics of Computation, 82:1401-1432, 2013
    
    \bibitem{Bona1} J. L. Bona, M. Chen, J.-C. Saut, \emph{Boussinesq equations and 
    other systems for small-amplitude long waves in nonlinear dispersive media: 
    I. Derivation and Linear Theory},    Journal of Nonlinear Science, 12:283-318, 2002
    
    \bibitem{Bona2} J. L. Bona, M. Chen, J.-C. Saut, \emph{Boussinesq equations and 
    other systems for small-amplitude long waves in nonlinear dispersive media: 
    II. The nonlinear theory}, Nonlinearity, 17:925-952, 2004
    
    \bibitem{BDM} J. L. Bona, V. A. Dougalis, D. E. Mitsotakis, \emph{Numerical solutions of KdV-KdV systems of Boussinesq equations I. The numerical scheme and generalized solitary 
      waves}, Mathematics and Computers in Simulation, 74:214-228, 2007
      
      
    \bibitem{Boussinesq1871} J.~Boussinesq, \emph{Th\'{e}orie de l'intumescence liquide appel\'{e}e 
    onde solitaire ou de translation se propageant dans un canal rectangulaire}, { Comptes Rendus de l'Acadmie de Sciences}, 72:755--759, 1871

\bibitem{JX} J. Buli, Y. Xing, \emph{Local discontinuous Galerkin methods for the Boussinessq coupled BBM system}, 
     Journal of Scientific Computing, 75:536-559, 2018
     
   \bibitem{CB} C. Burtea, C. Court\`{e}s, \emph{Discreet energy 
     estimates for the abcd-systems }, Communications in Mathematical Sciences, 
     17:243-298, 2019
     
      \bibitem{PC} P. Ciarlet, \emph{The finite element method for elliptic problem}, 
      North Holland, 1975
      
   \bibitem{MC} M. Chen, \emph{Exact traveling-wave solutions to bidirectional wave 
     equations}, International Journal of Theoretical Physics, 37:1547-1567, 1998
     
    

\bibitem{CHS1990}
B.~Cockburn, S.~Hou, C.-W. Shu.
\emph{The {R}unge-{K}utta local projection discontinuous {G}alerkin finite
  element method for conservation laws {IV}: the multidimensional case}, Mathematics of Computation, 54:545-581, 1990

\bibitem{CKS2000}
B.~Cockburn, G.~Karniadakis, C.-W. Shu.
\newblock The development of discontinuous galerkin methods.
\newblock In B.~Cockburn, G.~Karniadakis, and C.-W. Shu, editors, {\em
  Discontinuous Galerkin Methods: Theory, Computation and Applications}, pages
  3--50. Lecture Notes in Computational Science and Engineering, Part I:
  Overview, volume 11, Springer, 2000

\bibitem{CLS1989}
B.~Cockburn, S.-Y. Lin, C.-W. Shu,
\emph{{TVB} {R}unge-{K}utta local projection discontinuous {G}alerkin
  finite element method for conservation laws {III}: one dimensional systems,}
 Journal of Computational Physics, 84:90-113, 1989

\bibitem{CS1989}
B.~Cockburn, C.-W. Shu,
\emph{{TVB} {R}unge-{K}utta local projection discontinuous {G}alerkin
  finite element method for conservation laws {II}: general framework}, Mathematics of Computation, 52:411-435, 1989
    
    \bibitem{CC} B. Cockburn, C.-W. Shu, \emph{The local discontinuous Galerkin finite element method for convection-diffusion systems},
      SIAM Journal on Numerical Analysis, 35:2440-2463, 1998

     \bibitem{SSP} S. Gottlieb, C.-W. Shu, E. Tadmor, \emph{Strong stability-preserving high-order time discretization methods}, SIAM Review, 43:89-112, 2001
     
     \bibitem{HX2014} C. Hufford, Y. Xing, \emph{Superconvergence of the local discontinuous Galerkin method for the linearized Korteweg-de Vries equation}, Journal of Computational and Applied Mathematics, 255:441-455, 2014

     \bibitem{KX2016} O. Karakashian, Y. Xing, \emph{A posteriori error estimates for conservative local discontinuous Galerkin methods for the Generalized Korteweg-de Vries equation}, Communications in Computational Physics, 20:250-278, 2016
     
      \bibitem{XCY} X. H. Li, C.-W. Shu, Y. Yang, \emph{Local discontinuous Galerkin method for the Keller-Segel chemotaxis model},
      Journal of Scientific Computing, 73:943-967, 2017
   
 \bibitem{XL} X. Li, Y. Xing, C.-S. Chou, \emph{Optimal energy conserving and energy dissipative local discontinuous Galerkin methods for the Benjamin-Bona-Mahony equation}, Journal of Scientific Computing, 83:17, 2020
     
 \bibitem{LSXC2020} X. Li, W. Sun, Y. Xing, C.-S. Chou, \emph{Energy conserving local discontinuous Galerkin methods for the improved Boussinesq equation}, Journal of Computational Physics, 401:109002, 2020
     
     \bibitem{QZ} J. Luo, C.-W. Shu, Q. Zhang, \emph{A priori error estimates to smooth solutions of the
third order Runge-Kutta discontinuous Galerkin method for symmetrizable systems of conservation laws}, ESAIM: M2AN, 49:991-1018, 2015

      \bibitem{Peregrine} D. H. Peregrine, \emph{Calculations of the development of an undular bore}, Journal of Fluid Mechanics, 25:321-330, 1966
     
     \bibitem{YS} Y. Xu, C.-W. Shu, \emph{Error estimates of the semi-discrete local discontinuous Galerkin method for nonlinear convection-diffusion and KdV equations}, 
    Computer Methods in Appllied Mechanics and Engineering, 196:3805-3822, 2007
     
      \bibitem{XS2010} Y. Xu, C.-W. Shu, \emph{Local discontinuous {G}alerkin methods for high-order time-dependent partial differential equations},
   Communications in Computational Physics, 7:1-46, 2010
      
      \bibitem{XS2012} Y. Xu, C.-W. Shu, 
      \emph{Optimal error estimates of the semi-discrete local discontinuous Galerkin methods for high order wave equations}, 
      SIAM Journal on Numerical Analysis, 50:79-104, 2012

\end{thebibliography}
\end{document}